\documentclass[12pt]{amsart}
\usepackage{amsmath}
\usepackage{amssymb}
\usepackage{amsfonts}
\usepackage{mathrsfs}
\usepackage{mathtools}

   \newcommand{\Cdb}{\mbox{$\mathbb{C}$}}

   \newcommand{\Ndb}{\mbox{$\mathbb{N}$}}

   \newcommand{\Rdb}{\mbox{$\mathbb{R}$}}

   \newcommand{\Zdb}{\mbox{$\mathbb{Z}$}}

   \newcommand{\A}{\mbox{${\mathcal A}$}}

   \newcommand{\E}{\mbox{${\mathcal E}$}}

   \renewcommand{\H}{\mbox{${\mathcal H}$}}

    \newcommand{\M}{\mbox{${\mathcal M}$}}

\newcommand{\norm}[1]{\Vert#1\Vert}
\newcommand{\bignorm}[1]{\bigl\Vert#1\bigr\Vert}

\newtheorem{theorem}{Theorem}[section]
\newtheorem{lemma}[theorem]{Lemma}
\newtheorem{corollary}[theorem]{Corollary}
\newtheorem{proposition}[theorem]{Proposition}
\newtheorem{definition}[theorem]{Definition}
\theoremstyle{remark}
\newtheorem{remark}[theorem]{\bf Remark}
\theoremstyle{definition}

\numberwithin{equation}{section}

\begin{document}

\title[Absolutely dilatable Schur multipliers]
{A characterization of absolutely dilatable Schur multipliers}

\author[C. Duquet]{Charles Duquet}
\email{charles.duquet@univ-fcomte.fr}
\address{Laboratoire de Math\'ematiques de Besan\c con,  Universit\'e de Franche-Comt\'e,
16 Route de Gray, 25030 Besan\c{c}on Cedex, FRANCE}

\author[C. Le Merdy]{Christian Le Merdy}
\email{clemerdy@univ-fcomte.fr}
\address{Laboratoire de Math\'ematiques de Besan\c con,  Universit\'e de Franche-Comt\'e,
16 Route de Gray, 25030 Besan\c{c}on Cedex, FRANCE}

\date{\today}

\maketitle

\begin{abstract} Let $M$ be a von Neumann algebra equipped with
a normal semi-finite faithful trace (nsf trace in short) and let $T\colon M
\to M$ be a contraction. We
say that $T$ is absolutely dilatable 
if there exist another 
von Neumann algebra $M'$
equipped with a nsf trace,
a $w^*$-continuous trace preserving unital
$*$-homomorphism $J\colon M\to M'$
and a trace preserving $*$-automomorphism $U\colon M'\to M'$ such that 
$T^k=E U^k J$ for all integer $k\geq 0$, where $E\colon M'\to M$ is the 
conditional expectation associated with $J$. 
Given a $\sigma$-finite measure space
$(\Omega,\mu)$, we characterize bounded Schur multipliers $\phi\in L^\infty(\Omega^2)$ 
such that the Schur multiplication operator $T_\phi\colon B(L^2(\Omega))\to B(L^2(\Omega))$
is absolutely dilatable. In the separable case, they are characterized by the
existence of a von Neumann algebra $N$ with a separable predual, equipped with a normalized
normal faithful trace $\tau_N$,  and  a
$w^*$-measurable essentially bounded function $d\colon\Omega\to N$ such that
$\phi(s,t)=\tau_N(d(s)^*d(t))$ for almost every $(s,t)\in\Omega^2$.
\end{abstract}



\section{Introduction} 
Akcoglu’s dilation theorem \cite{A,AS} asserts that for any $1<p<\infty$, for any measure space
$(\Omega,\mu)$ and for any positive contraction $T\colon L^p(\Omega)\to L^p(\Omega)$, 
there exist
another measure space $(\Omega',\mu')$, an isometry $J\colon 
L^p(\Omega)\to L^p(\Omega')$, an isometric isomorphism
$U\colon L^p(\Omega')\to L^p(\Omega')$ and a contraction  $Q\colon L^p(\Omega')\to
L^p(\Omega)$ such that $T^k=QU^kJ$ for all integer $k\geq 0$. This result is of paramount
importance in various topics concerning $L^p$-operators. On the one hand, it was 
at the root of the sharpest results about $H^\infty$ functional calculus for either
single operators or semigroups of operators 
on $L^p$-spaces. See in particular \cite{AFL, CP, HP}
as well as \cite{KW}, where Akcoglu’s theorem is combined with Fourier multiplier
theorems to provide a bounded $H^\infty$ functional calculus for various differential
operators. On the other hand, it plays a key role in the harmonic analysis of semigroups,
ergodic theory and maximal inequalities on $L^p$-spaces, see in particular
Cowling's seminal paper
\cite{Cow}, the more recent articles \cite{CCL,
LMX, Xu}, as well as \cite[Section I]{HVVW2} and
the references therein.

In the last two decades, operator theory and harmonic analysis on 
non-commutative $L^p$-spaces gained a lot
of interest and many classical results concerning $L^p(\Omega)$-spaces have been transferred
to the non-commutative setting (see \cite{JMP} for an introduction). 
However Akcoglu’s theorem does not extend to 
non-commutative $L^p$-spaces. More precisely, let $M$ be a von Neumann 
algebra equipped with a normal semi-finite faithful trace $\tau$. In this situation, we say that
$(M,\tau)$ is a tracial von Neumann algebra. Let $1<p<\infty$ and let
$L^p(M)$ be the associated non-commutative $L^p$-space. Let us say  that a 
contraction $T\colon L^p(M)\to L^p(M)$ admits an isometric 
$p$-dilation if there exist another  tracial von Neumann algebra $(M',\tau')$,
an isometry $J\colon L^p(M)\to L^p(M')$,  an isometric isomorphism
$U\colon L^p(M')\to L^p(M')$  and a contraction $Q \colon  L^p(M')\to L^p(M)$  such that 
$T^k=QU^kJ$ for all integer $k\geq 0$. 
Then it was shown in \cite{JLM} that not all positive contractions 
$T\colon L^p(M)\to L^p(M)$ admit an isometric 
$p$-dilation. After this negative result was published, it became important 
to exhibit classes of positive contractions
on non-commutative $L^p$-spaces which do admit an isometric 
$p$-dilation. Such examples appeared in \cite[Chapters 9-10]{JLX},
including contractions on Clifford algebras (more generally, on
$q$-deformed von Neumann algebras) arising from second quantization,
and the non-commutative Poisson semigroup. See also 
\cite[Section 8]{LM} and \cite{Arh1,Arh2, Arh4, Arh3} for various 
dilation results for classes of either Schur or (non-commutative) 
Fourier multipliers. We also mention \cite{HRW} for a recent isometric
$p$-dilation property for Lamperti type operators on non-commutative $L^p$-spaces,
which leads to a (non-commutative) ergodic maximal inequality.

We are interested in the notion of absolute dilation. 
Consider $(M,\tau)$ as above and
let $T\colon M\to M$ be a contraction. We say that $T$ 
is absolutely dilatable (or admits an absolute dilation)
if there exist another tracial von Neumann algebra $(M',\tau')$, 
a $w^*$-continuous trace preserving unital
$*$-homomorphim $J\colon M\to M'$
and a trace preserving $*$-automomorphim $U\colon M'\to M'$ such that 
$T^k= E U^k J$ for all $k\geq 0$, where $E\colon M'\to M$ 
is the conditional expectation associated with $J$. If $M$ has 
a separable predual and the above property holds for some 
$M'$ with a separable predual, we say that  $T$ admits a separable
absolute dilation.
It turns out that if $T\colon M\to M$ 
is absolutely dilatable, then for all $1<p<\infty$, $T$ extends 
to a positive contraction on $L^p(M)$
which admits an isometric 
$p$-dilation (see Proposition \ref{2Abs-p}). Absolutely dilatable operators
were formally introduced in \cite{Du}. However they appear implicitly in
\cite{Ku, HM} (without any reference to non-commutative $L^p$-spaces)
as well as in  \cite{Arh1, Arh2, JLX}, where they are used for the purpose of 
obtaining isometric
$p$-dilations. Indeed, many known constructions of isometric 
$p$-dilations for specific classes
of $L^p(M)$-operators arise from  absolute dilations. Also,
absolute dilations are closely related
to the so-called factorizable operators \cite{AD, HM, JRS, Ric}.
As noticed in Remark \ref{2Dil2}, (4), absolute dilations generalize 
factorizability to non-normalized tracial von Neumann algebras.

This paper is concerned with Schur multipliers acting on 
$B(L^2(\Omega))$, the space of bounded operators on 
the $L^2$-space associated with a $\sigma$-finite measure space $(\Omega,\mu)$.
The study of bounded Schur multipliers in this setting goes back at least 
to Haagerup \cite{Haa} and Peller \cite{Pe1} (see also Spronk \cite{Spronk}).
We refer to Section \ref{S-Schur} for definitions
and background. Schur multipliers are fundamental objects in various parts
of non-commutative analysis. In particular, they play a key role in the theory
of multiple operator integrals and their applications to perturbation theory (see
\cite{Pe2} for a survey), for the study of approximation properties of operator algebras
(see e.g. \cite{DL}) and, using transference results, for the study of Fourier
multipliers. We refer to \cite{CGPT, PRD} for recent  outstanding advances on
Schur multipliers and to \cite{AK} for applications of dilations of Schur multipliers.

The papers \cite{Arh1,Du} show that a positive unital
Schur multiplier operator $T_\phi\colon B(L^2(\Omega))\to B(L^2(\Omega))$ 
associated with a real-valued $\phi\in L^\infty(\Omega^2)$ 
is absolutely dilatable. Further the Haagerup-Musat paper
\cite{HM} provides a characterization of
absolutely dilatable Schur multipliers in the finite
dimensional setting, see Remark \ref{5FD}. The objective of the present paper is to establish
a  characterization of
absolutely dilatable Schur multipliers $B(L^2(\Omega))\to B(L^2(\Omega))$
in the general case.

Let us start with the discrete case. For any $n\geq 1$, let $M_n$
denote the space of $n\times n$ matrices with complex entries. Let 
${\frak m}=\{m_{ij}\}_{i,j\geq 1}$ be a family of complex numbers. Recall that  
${\frak m}$ is a bounded Schur multiplier if there exists a (necessarily unique)
$w^*$-continuous operator $T_{\frak m}\colon B(\ell^2)\to B(\ell^2)$
such that $T_{\frak m}\bigl([a_{ij}]\bigr) = [m_{ij}a_{ij}]$
for all $n\geq 1$ and all $[a_{ij}]\in M_n\subset B(\ell^2)$.
We show that 
$T_{\frak m}$ admits an absolute dilation if and only if  
there exist a normalized
tracial von Neumann algebra
$(N,\tau_N)$ with separable predual and a sequence $(d_k)_{k\geq 1}$
of unitaries of $N$ such that 
\begin{equation}\label{1TraceForm0}
m_{ij}=\tau_N(d_i^*d_j),\qquad i,j\geq 1.
\end{equation}
In this case, $T_{\frak m}$ actually admits a separable absolute dilation, see Corollary \ref{5Discrete}.

In the general separable case, spaces of bounded families
with entries in a von Neumann algebra $N$ with a separable predual are replaced
by  spaces $L^\infty_\sigma(\Omega;N)$ consisting of 
(classes of) $w^*$-measurable essentially bounded  functions $\Omega\to N$.
These spaces are used e.g. 
in \cite[Theorem 1.22.13]{Sak} and we have an identification
$L^\infty_\sigma(\Omega;N)\simeq L^\infty(\Omega)\overline{\otimes} N$, see 
Section \ref{SNormal} for details and information.

In the sequel, we say that $(\Omega,\mu)$ is separable when
the Hilbert space $L^2(\Omega)$
is separable. Our main result is the following. Note that is this general case, (\ref{1TraceForm}) 
replaces (\ref{1TraceForm0}).

\begin{theorem}\label{1Main}
Assume that $(\Omega,\mu)$ is separable and let 
 $\phi\in L^\infty(\Omega^2)$. The following assertions are equivalent.
\begin{itemize}
\item [(i)] The function $\phi$ is a bounded Schur multiplier and 
$T_\phi\colon B(L^2(\Omega))\to B(L^2(\Omega))$ admits a separable absolute dilation.
\item [(ii)] There exist a normalized tracial von Neumann algebra $(N,\tau_N)$ 
with a separable predual and a unitary $d\in L^\infty_\sigma(\Omega; N)$
such that 
\begin{equation}\label{1TraceForm}
\phi(s,t) =\tau_N\bigl(d(s)^*d(t)\bigr), 
\qquad{\rm for\  a.e.}\ (s,t)\in\Omega^2.
\end{equation}
\end{itemize}
\end{theorem}

Sections 2-4 provide all the necessary background and preliminary results on dilations, 
$L^\infty_\sigma(\Omega;\bullet)$-spaces, tensor products and
Schur multipliers. The proof of Theorem \ref{1Main} is given in Section \ref{SProof}.
It relies on various preparatory results presented in 
Section \ref{Prep}. The final Section \ref{S-NSC} is dedicated to the
non separable case. We give a version of Theorem \ref{1Main} 
for Schur multipliers which admit a (not necessarily separable) absolute
dilation. We also discuss the discrete case in the last two sections.

\section{Dilations}\label{SDil}
Let $M$ be a von Neumann algebra. We let $M^+$ denote its positive
part and unless otherwise stated, we let $1_M$ denote its unit.
If $M$ is equipped with a normal semi-finite faithful trace
$\tau_M\colon M^+\to[0,\infty]$, we say that $(M,\tau_M)$ is 
a tracial von Neumann. We further say that $(M,\tau_M)$  is normalized if 
$\tau_M(1_M)=1$. In this case, $\tau_M$ is finite and extends
to a normal tracial state on $M$.

Given any tracial von Neumann algebra $(M,\tau_M)$ acting on some Hilbert 
space $H$,  let 
$L^0(M,\tau_M)$ denote the $*$-algebra of all possibly unbounded
operators on $H$ which are affiliated with $M$. Let
$\widetilde{\tau_M}\colon L^0(M,\tau_M)^+\to[0,\infty]$ denote the natural 
extension of $\tau_M$ to the positive part of $L^0(M,\tau_M)$.
Then for all 
$1\leq p<\infty$, the non-commutative $L^p$-space $L^p(M,\tau_M)$
is defined as the space of all $x\in L^0(M,\tau_M)$ such that 
$\widetilde{\tau_M}(\vert x\vert^p)<\infty$. This is a Banach space for the norm
$\norm{x}_p=(\widetilde{\tau_M}(\vert x\vert^p))^{\frac{1}{p}}$.
For convenience, we write
$L^\infty(M,\tau_M)=M$. It turns out that $M\cap L^p(M,\tau_M)$
is dense in $L^p(M,\tau_M)$.

The trace $\tau_M$ is finite on $L^1(M,\tau_M)\cap M^+$
and extends to a functional on $L^1(M,\tau_M)$, still
denoted by $\tau_M$. Let $1\leq p<\infty$
and let $p'=\frac{p}{p-1}$ be its conjugate index. For any 
$x\in L^p(M,\tau_M)$ and any $y\in L^{p'}(M,\tau_M)$, the product
$xy$ belongs to $L^1(M,\tau_M)$  and we have a H\"older inequality $\vert \tau_M(xy)
\vert \leq  \norm{x}_p\norm{y}_{p'}$. 
We further have an isometric
isomorphism 
$$
L^{p}(M,\tau_M)^*\simeq
L^{p'}(M,\tau_M),
$$
for the duality pairing given by
$\langle y,x\rangle = \tau_M(xy)$ for all  $x\in L^{p}(M,\tau_M)$
and $y\in L^{p'}(M,\tau_M)$. In particular, 
\begin{equation}\label{2Dual}
M\simeq
L^1(M,\tau_M)^*.
\end{equation}

The space $L^2(M,\tau_M)$ is a Hilbert space, with inner product given by
$(x\vert y)=\tau_M(xy^*)$ for all $x,y\in L^2(M,\tau_M)$.

Let $H$ be a Hilbert space and let $B(H)$ denote the von Neumann
algebra of all bounded operators on $H$.  
Let ${\rm tr}\colon B(H)^+\to[0,\infty]$ denote the usual trace. 
Then for all $1\leq p<\infty$,
$L^{p}(B(H), {\rm tr})$ is  equal to the Schatten $p$-class  $S^p(H)$.  
We note that $S^2(H)$ coincides with the space 
of Hilbert-Schmidt operators on $H$.

The reader is referred to \cite[Section V.2]{Tak}  for information on traces and 
to the survey paper \cite{PX} and to \cite[Chapter I]{Terp} and \cite{FK}  for details on non-commutative $L^p$-spaces
(see also \cite[Section 4.3]{Hiai} and \cite{Kos}).

Let $(M_1,\tau_{M_1})$ and $(M_2,\tau_{M_2})$ be two tracial von Neumann algebras
and let $U\colon M_1\to M_2$ be a positive map, that is,
$U(M_{1}^+)\subset M_{2}^+$. We say that $U$ is trace preserving if
$\tau_{M_2}\circ U = \tau_{M_1}$ on $M_{1}^+$. If
$U$ is a trace preserving $*$-homomorphim, then $U$ is 1-1 (because $\tau_{M_1}$
is faithful) and it is well-known that  for all $1\leq p<\infty$,
the restriction of $U$ to $M\cap L^p(M,\tau_M)$
extends to an isometry 
$$
U_p\colon L^p(M_1,\tau_{M_1})\longrightarrow L^p(M_2,\tau_{M_2}).
$$
Moreover $U_p$ is onto if $U$ is onto. We refer to \cite[Theorem 2] {Y} for this fact and to
\cite[Lemma 1.1]{JX} for complementary results.

\begin{definition}\label{2Dil0}
Let $(M,\tau_M)$ be a tracial von Neumann algebra. 
We say that an operator  $T\colon M\to M$
is {\bf absolutely dilatable}, or admits an
{\bf absolute dilation}, if there exist another tracial von Neumann algebra
$(\M,\tau_{\mathcal M})$, a $w^*$-continuous trace preserving unital
$*$-homomorphim $J\colon M\to \M$
and a trace preserving $*$-automomorphim $U\colon \M\to \M$ such that 
\begin{equation}\label{2Dil1}
T^k=J_1^* U^k J,\qquad k\geq 0.
\end{equation} 
Here $J_1\colon L^1(M,\tau_M)\to L^1(\M,\tau_{\mathcal M})$ stands for the isometric
extension of $J$ and $J_1^*\colon\M\to M$ is defined using
(\ref{2Dual}).

If $M$ has a separable predual, we say that $T\colon M\to M$
admits a {\bf separable absolute dilation} if the above property holds with some
$\M$ with a separable predual.
\end{definition}

\begin{remark}\label{2Dil2}
Here are a few comments about the above definition.

\smallskip
(1) Let $J\colon M\to \M$ be a  trace preserving
$*$-homomorphim. Then for all $x\in M$ and all
$y\in M\cap L^1(M,\tau_M)$, we have 
$$
\langle x,y\rangle = \tau_M(xy)=\tau_{\mathcal M}(J(xy))
=\tau_{\mathcal M}(J(x)J_1(y)) = \langle J(x), J_1(y)\rangle.
$$
This implies that $J_1^*J=I_{M}$, the identity operator on $M$. Hence the relation 
(\ref{2Dil1}) for $k=0$ is automatic. Thus (\ref{2Dil1}) is significant for $k\geq 1$ only.
See Remark \ref{5k1} for more about this.

\smallskip
(2) If $J\colon M\to \M$ is  a $w^*$-continuous trace preserving  unital 
$*$-homomorphim, then $M$ can be regarded (through $J$) as a von Neumann subalgebra
of $\M$ whose trace $\tau_M$ is the restriction of $\tau_{\mathcal M}$. In this respect,
$J_1^*\colon\M \to M$ is the natural conditional expectation onto $M$
(see e.g. \cite[Proposition V.2.36]{Tak}). In particular, $J_1^*$ is positive, 
trace preserving and unital.

\smallskip 
(3) If $T\colon M\to M$ is absolutely dilatable, then $T$ is necessarily positive,
trace preserving and unital. This follows from the factorization $T=J_1^*UJ$ and 
the previous paragraph.

\smallskip 
(4) Assume that  $(M,\tau_M)$ is normalized. According to \cite[Theorem 4.4]{HM}, 
an operator  $T\colon M\to M$ is factorizable with respect to $\tau_M$, in the sense of \cite{AD}, if and only if
it is absolutely dilatable and the tracial von Neumann algebra
$(\M,\tau_{\mathcal M})$ in Definition \ref{2Dil0} can be chosen normalized. Thus,
absolute dilatability should be regarded as a non-normalized version of factorizability.

\smallskip 
(5) A von Neumann algebra $M$ has a separable predual if and only if there exists a separable Hibert space
$H$ such that $M\subset B(H)$ as a von Neumann subalgebra.
\end{remark}

Let $(M,\tau_M)$ be a tracial von Neumann algebra, let 
$1\leq p<\infty$ and let $T_0\colon L^p(M,\tau_M)\to L^p(M,\tau_M)$
be a contraction. As indicated in the introduction, we say that $T_0$ admits an isometric $p$-dilation if there
exist another tracial von Neumann algebra
$(\M,\tau_{\mathcal M})$, an isometry $J_0\colon 
L^p(M,\tau_M)\to L^p(\M,\tau_{\mathcal M})$, an isometric isomorphism
$U_0\colon L^p(\M,\tau_{\mathcal M})\to L^p(\M,\tau_{\mathcal M})$ 
and a contraction $Q_0\colon  L^p(\M,\tau_{\mathcal M})
\to L^p(M,\tau_M)$ such that $T_0^k=Q_0U_0^kJ_0$ for all $k\geq 0$.

\begin{proposition}\label{2Abs-p}
Let $T\colon M\to M$ be  an absolutely dilatable operator and let 
$1\leq p<\infty$.  Then the restriction of $T$ to $M\cap L^p(M)$ extends to a positive contraction 
$T_0\colon L^p(M,\tau_M)\to L^p(M,\tau_M)$ and the latter 
admits an isometric $p$-dilation.
\end{proposition}

\begin{proof}
Let $(\M,\tau_{\mathcal M})$, $J$ and $U$ as in Definition \ref{2Dil0}.
Then $T_0= J_{p'}^*U_pJ_p$ is a positive contraction which 
coincides with $T$ on $M\cap L^p(M)$.
Furthermore, we have
$T_0^k= J_{p'}^*U_p^kJ_p$ for all $k\geq 0$. Since 
$J_p$ and $U_p$ are isometries, we obtain that 
$T_0$ admits an isometric $p$-dilation.
\end{proof}

\section{$L^\infty_\sigma$-spaces and normal tensor products}\label{SNormal}
Let $(\Omega,\mu)$ be a $\sigma$-finite measure space. For any $1\leq p\leq \infty$,
we simply let $L^p(\Omega)$ denote the associated $L^p$-space
$L^p(\Omega,\mu)$. In integration formulas, $d\mu(t)$ will be abreviated 
to $dt$.

Given any Banach space $Y$ and any $1 \leq p\leq \infty$, we let $L^p(\Omega;Y)$
denote the Bochner space of all measurable functions $F\colon\Omega\to Y$, 
defined up
to almost everywhere zero functions, such that the norm function $t\mapsto\norm{F(t)}_Y$
belongs to $L^p(\Omega)$. 
This is a Banach space for the norm $\norm{F}_p$, 
defined as the $L^p(\Omega)$-norm of $\norm{F(\,\cdotp)}_Y$ 
(see \cite[Chapters I and II]{DU}). 
We note that if $p\not=\infty$, $L^p(\Omega)\otimes Y$
is a dense subspace of $L^p(\Omega;Y)$. Furthermore if
$Y=H$ is a Hilbert space, then the natural embedding of 
$L^2(\Omega)\otimes H$ into $L^2(\Omega;H)$
extends to a unitary identification
\begin{equation}\label{Hilbert-TP}
L^2(\Omega;H) \simeq L^2(\Omega)\stackrel{2}{\otimes} H,
\end{equation}
where $\stackrel{2}{\otimes}$ stands for the Hilbertian tensor product.

Let $X$ be a separable Banach space. We say that a function $d\colon\Omega\to X^*$ 
is $w^*$-measurable if for all $x\in X$, the scalar function $t \mapsto 
\langle d(t),x\rangle$ is measurable on $\Omega$. The latter function is denoted by 
$\langle d,x\rangle$ for simplicity.
If $d\colon\Omega\to X^*$ 
is $w^*$-measurable, then the norm function $t\mapsto \norm{d(t)}$ is measurable on  $\Omega$. 
Indeed if $(x_n)_{n\geq 1}$ is a dense sequence in the unit ball
of $X$, then 
$\norm{d(\,\cdotp)} =\sup_n\vert\langle d,x_n\rangle\vert$.
We let $L^\infty_\sigma(\Omega;X^*)$ denote the space of all
$w^*$-measurable functions $d\colon\Omega\to X^*$, defined up
to almost everywhere zero functions, such that 
$\norm{d(\,\cdotp)}$ is essentially bounded (in this case, we imply
say that $d$ is essentially bounded).
This is a Banach space
for the norm 
$$
\norm{d}_\infty : = \bignorm{\norm{d(\,\cdotp)}}_{L^\infty(\Omega)}.
$$
Here the subscript $\sigma$ refers to the $w^*$-topology $\sigma(X^*,X)$ on 
$X^*$.

Let $F\in L^1(\Omega;X)$ and let $d\in L^\infty_\sigma(\Omega;X^*)$.
Using the fact that $L^1(\Omega)\otimes 
X$ is dense in $L^1(\Omega;X)$, it is easy to see that the function 
$t\mapsto \langle d(t), F(t)\rangle$
is integrable and that 
\begin{equation}\label{3Pairing}
\langle d,F\rangle : = \int_\Omega\langle d(t), F(t)\rangle\, dt
\end{equation}
satisfies $\vert \langle d,F\rangle\vert\leq\norm{F}_{1}
\norm{d}_{\infty}$. 
Consequently,
(\ref{3Pairing}) defines a duality pairing between 
 $L^1(\Omega;X)$ and $L^\infty_\sigma(\Omega;X^*)$. 
It turns out that (\ref{3Pairing}) extends to an isometric isomorphism
\begin{equation}\label{3Sigma}
L^\infty_\sigma(\Omega;X^*) \simeq L^1(\Omega;X)^*.
\end{equation}
This result is known as the Dunford-Pettis theorem 
\cite[Theorem 2.1.6]{DP}.
See \cite[Theorem 11]{CLS} for a simple proof.

It is clear that
$$
L^\infty(\Omega;X^*)\subset L^\infty_\sigma(\Omega;X^*)
$$
According to \cite[Theorem IV.1.1]{DU}, this inclusion is an equality
if and only if $X^*$ has the so-called  Radon-Nikodym property. 
We will use the fact that Hilbert spaces have this property \cite[Corollary IV.14]{DU}. 
Thus, for any separable Hilbert space $\H$, we have
\begin{equation}\label{3RN}
L^\infty(\Omega;\H) = L^\infty_\sigma(\Omega;\H).
\end{equation}
Note that more generally, all reflexive Banach spaces have 
the Radon-Nikodym property. See \cite[Section VII.7]{DU} for more information.

Let $W\subset X^*$ be any $w^*$-closed subspace. Then $W$ 
has a natural predual, equal to the quotient space $X/W_\perp$, and the latter
is separable. This allows to define
$L^\infty_\sigma(\Omega;W)$. It is clear that 
$$
L^\infty_\sigma(\Omega;W) \subset L^\infty_\sigma(\Omega;X^*)
$$
isometrically. 
Furthermore, $L^1(\Omega;W_\perp)\subset L^1(\Omega;X)$ isometrically  and
under the identification
(\ref{3Sigma}), we have
\begin{equation}\label{3Inclusion2}
L^\infty_\sigma(\Omega;W) = L^1(\Omega;W_\perp)^\perp.
\end{equation}

Given any von Neumann algebras $M_1,M_2$, we let $M_1\overline{\otimes}M_2$
denote their normal tensor product. We will assume that the reader is familiar 
with the normal tensor product and its basic properties, for which
we refer either to \cite[Section IV.5]{Tak} (where it is called the $W^*$-tensor product)
or to \cite[Section 11.2]{KR}, see also \cite[Section 1.22]{Sak}.
We recall  a few facts that can be found in the latter references.

First assume that $M_1\subset B(H_1)$ and $M_2\subset B(H_2)$ as von Neumann
subalgebras, for some Hilbert spaces $H_1,H_2$. Then 
$$
M_1\overline{\otimes}M_2\subset B(H_1)\overline{\otimes}B(H_2)\simeq
B(H_1 \stackrel{2}{\otimes} H_2).
$$
More precisely, if we regard $M_1\otimes M_2\subset B(H_1 \stackrel{2}{\otimes} H_2)$
in the usual way, then $M_1\overline{\otimes}M_2$ is the $w^*$-closure
of $M_1\otimes M_2$ in $B(H_1 \stackrel{2}{\otimes} H_2)$.

Second, consider $\eta_1\in (M_1)_*$ and $\eta_2\in (M_2)_*$. 
Let $I_{M_1}$ (resp. $I_{M_2}$) denote the identity operator on
$M_1$ (resp. $M_2$).
Then the tensor map $\eta_1\otimes I_{M_2}\colon 
M_1 \otimes M_2\to M_2$ uniquely extends to a $w^*$-continuous
operator $\eta_1\overline{\otimes} I_{M_2}\colon 
M_1\overline{\otimes}M_2\to M_2$. Likewise, the tensor map $I_{M_1}\otimes \eta_2$ 
uniquely extends to a $w^*$-continuous
operator $I_{M_1}\overline{\otimes} \eta_2\colon 
M_1\overline{\otimes}M_2\to M_1$. Further 
$\eta_1\otimes \eta_2$ uniquely extends to a 
$w^*$-continuous functional on 
$M_1\overline{\otimes}M_2$. We still denote this extension by $\eta_1\otimes\eta_2$.
It is plain that 
$$
\eta_1\otimes \eta_2 =\eta_1\circ(I_{M_1}\overline{\otimes}\eta_2)
=\eta_2\circ(\eta_1\overline{\otimes}I_{M_2}).
$$

Third, the above definition of $\eta_1\otimes\eta_2\in (M_1\overline{\otimes}M_2)_*$
yields a natural embedding
\begin{equation}\label{3Predual}
(M_1)_*\otimes (M_2)_*\subset (M_1\overline{\otimes}M_2)_*.
\end{equation}

\begin{remark}\label{3Slice}
The operators $\eta_1\overline{\otimes} I_{M_2}$ and $I_{M_1}\overline{\otimes}\eta_2$
considered above are usually called the slice maps associated with $\eta_1$ and $\eta_2$,
respectively.

Assume  $M_1\subset B(H_1)$ and $M_2\subset B(H_2)$ as above and let 
$w\in  B(H_1)\overline{\otimes}B(H_2)$. Then 
$w\in M_1\overline{\otimes}M_2$ if and only if 
$$
\bigl(I_{B(H_1)}\overline{\otimes} \eta_2\bigr)(w)\in M_1
\qquad\hbox{and}\qquad
\bigl(\eta_1\overline{\otimes} I_{B(H_2)}\bigr)(w)\in M_2,
$$
for all $\eta_1\in B(H_1)_*$ and all $\eta_2\in B(H_2)_*$. This is the
so-called slice map property of von Neumann algebras, see \cite{Tom}.
In this paper, we will use this property only when $M_2=B(H_2)$, in which case the proof is simple 
(see e.g. \cite[Theorem 2.4]{EKR} and its proof).
\end{remark}

We conclude this section
with  a natural connection between $L^\infty_\sigma$-spaces 
and normal tensor products.

\begin{lemma}\label{3Sakai}
Let $M$ be a von Neumann algebra with a separable predual.
\begin{itemize}
\item [(1)] For any $c,d\in L^\infty_\sigma(\Omega;M)$, the functions
$t\mapsto c(t)d(t)$ and $t\mapsto d(t)^*$ are $w^*$-measurable. Furthermore,
$L^\infty_\sigma(\Omega;M)$ is a von Neumann algebra for the pointwise product and the pointwise involution.
\item [(2)] The natural embedding 
of $L^\infty(\Omega)\otimes M$ into $L^\infty_\sigma(\Omega;M)$
extends to a von Neumann algebra identification
\begin{equation}\label{3Sakai1}
L^\infty_\sigma(\Omega;M)\simeq L^\infty(\Omega)\overline{\otimes}M.
\end{equation}
\end{itemize}
\end{lemma}

\begin{proof}
This follows from \cite[Theorem 1.22.13]{Sak} and its proof.
\end{proof}

In the sequel, we write
$$
d\sim D
$$
if $d\in L^\infty_\sigma(\Omega;M)$ and $D\in L^\infty(\Omega)\overline{\otimes}M$
are corresponding to each other under the identification
(\ref{3Sakai1}).

\begin{remark}\label{3Sakai2}
It is plain that 
for any $d\in L^\infty_\sigma(\Omega;M)$ and 
$D\in L^\infty(\Omega)\overline{\otimes}M$,
we have $d\sim D$
if and only if
$$
\int_\Omega f(t) \langle d(t),\eta\rangle\, dt\,
=[f\otimes\eta](D)
$$
for all $f\in L^1(\Omega)$ and all $\eta\in M_*$.
\end{remark}

\begin{lemma}\label{3x}
Assume that $M$ has a separable predual, let $d\in L^\infty_\sigma(\Omega;M)$ and let 
$D\in L^\infty(\Omega)\overline{\otimes}M$ such that $d\sim D$. Then for all
$x\in M$, we have 
$$
xd\sim  (1\otimes x)D,
$$
where
$xd\in L^\infty_\sigma(\Omega;M)$ is defined by $(xd)(t)=xd(t)$, 
and $1=1_{L^\infty(\Omega)}$.
\end{lemma}

\begin{proof}
This readily follows from Lemma \ref{3Sakai}, using the constant function $c(t)=x$.
\end{proof}

As a complement to Lemma \ref{3Sakai}, we note that
the natural embedding 
of $L^\infty(\Omega)\otimes L^\infty(\Omega)$ into 
$L^\infty(\Omega^2)$ extends to a von Neumann algebra identification
\begin{equation}\label{3Omega2}
L^\infty(\Omega^2)\simeq L^\infty(\Omega)\overline{\otimes} L^\infty(\Omega).
\end{equation}
Indeed, the embedding of $L^1(\Omega)\otimes L^1(\Omega)$
into $L^1(\Omega^2)$ extends to an isometric identification
$L^1(\Omega;L^1(\Omega))\simeq L^1(\Omega^2)$, by Fubini's theorem. 
Hence (\ref{3Omega2}) follows from \cite[1.22.10--1.22.12]{Sak}.

\begin{remark}\label{TP-trace}
We finally recall the following classical fact, for which we refer to
\cite[Proposition V.2.14]{Tak}. Let $(M,\tau_M)$
be any tracial von Neumann algebra. 
Let $I$ be an index set and 
let $(e_i)_{i\in I}$ denote the standard Hilbertian basis of $\ell^2_I$.
Then ${\rm tr}\otimes\tau_M$ uniquely
extends to a normal semi-finite faithful trace 
$$
{\rm tr}\overline{\otimes}\tau_{M}\colon (B(\ell^2_I)
\overline{\otimes} M)^+\longrightarrow [0,\infty],
$$
which can be described as follows. 
Any $z\in B(\ell^2_I)\overline{\otimes} M$ may be naturally 
regarded as an infinite matrix $[z_{ij}]_{(i,j)\in I^2}$ with entries
$z_{ij}\in M$. If $z$ is positive, then each $z_{ii}$ is positive and we have
$$
\bigl[{\rm tr}\overline{\otimes}\tau_{M}\bigr](z)=\sum_{i\in I}
\tau_{M}(z_{ii}).
$$
\end{remark}

\section{Bounded Schur multipliers on $B(L^2(\Omega))$}\label{S-Schur}
Let $(\Omega,\mu)$ be a $\sigma$-finite measure space. 
For any  $\varphi\in L^2(\Omega^2)$, we let $S_\varphi\colon L^2(\Omega)\to L^2(\Omega)$
be the bounded operator defined by 
$$
[S_\varphi(h)](s) = \int_\Omega \varphi(s,t)h(t)\, dt,
\qquad h\in L^2(\Omega).
$$
We recall that $S_\varphi$ is a Hilbert-Schmidt operator,
that is, $S_\varphi\in S^2(L^2(\Omega))$. Further the mapping 
$\varphi\mapsto S_\varphi$ yields a unitary identification
\begin{equation}\label{2HS}
L^2(\Omega^2)\simeq S^2(L^2(\Omega)),
\end{equation}
see e.g. \cite[Theorem VI. 23]{RS}.

Recall that we let ${\rm tr}$ denote the usual trace on $B(L^2(\Omega))$. 
It follows from the above unitary identification and a simple polarization argument that for any
$\varphi_1,\varphi_2\in L^2(\Omega^2)$, we have
\begin{equation}\label{2Trace0}
{\rm tr}\bigl(S_{\varphi_1}S_{\varphi_2}\bigr)
=\int_{\Omega^2} \varphi_1(s,t)\varphi_2(t,s)\, dtds.
\end{equation}

Given any $u,v\in L^2(\Omega)$, we regard $u\otimes v$
as an element of $L^2(\Omega^2)$ using the convention
$(u\otimes v)(s,t)=u(t)v(s)$,  for $(s,t)\in\Omega^2$. 
In the sequel, the operator $S_{u\otimes v}$ will be abusively denoted
by $u\otimes v$. In other words, we use the notation
$u\otimes v$ to also denote the element of $B(L^2(\Omega))$ defined by
\begin{equation}\label{uv}
[u\otimes v](h) = \Bigl(\int_\Omega u(t)h(t)\, dt\Bigr) v,
\qquad h\in L^2(\Omega).
\end{equation}

Let $\phi\in L^\infty(\Omega^2)$. Using (\ref{2HS}), we define
$T_\phi\colon S^2(L^2(\Omega))\to S^2(L^2(\Omega))$
by setting 
$$
T_\phi(S_\varphi)=S_{\phi\varphi},\qquad
\varphi\in L^2(\Omega^2).
$$
This operator is called a Schur multiplication operator.
Applying (\ref{2Trace0}), we obtain that for all 
$u,v,a,b\in L^2(\Omega)$, we have
\begin{equation}\label{2Trace}
{\rm tr}\Bigl(\bigl[T_\phi(u\otimes v)\bigr]\bigl[a\otimes b\bigr]\Bigr)
=\int_{\Omega^2} \phi(s,t)u(t)v(s)a(s)b(t)\, dtds.
\end{equation}

Let $S^\infty(L^2(\Omega))$ denote the Banach space of all compact operators
on $L^2(\Omega)$ and recall that $S^2(L^2(\Omega))$ is a dense
subspace of $S^\infty(L^2(\Omega))$.
We say that $\phi$ is a bounded Schur multiplier if $T_\phi$ extends to a bounded
operator from $S^\infty(L^2(\Omega))$ into itself.  In this case, 
$T_\phi$ admits a unique $w^*$-continuous extension from $B(L^2(\Omega))$ into
itself. Indeed, $B((L^2(\Omega))$ is the second dual of $S^\infty((L^2(\Omega))$
(see \cite[Chapter 19]{Conway1}).
In the sequel, we  keep the notation
$$
T_\phi\colon B(L^2(\Omega))\longrightarrow B(L^2(\Omega))
$$
to denote this extension.
Schur multipliers as defined in this section
go back at least 
to Haagerup \cite{Haa} (see also \cite{Pe1, Spronk}).
We refer to \cite{Du} and the references therein for details and complementary results.

The following description of bounded Schur multipliers has a rather long history.
In the discrete case, part (1) of Theorem \ref{2Description} 
was stated by Paulsen in \cite[Corollary 8.8]{Pau} and
by Pisier in \cite[Theorem 5.1]{Pis}, who refers himself 
to some earlier work of Grothendieck. 
For the general case of this statement, we refer to Haagerup
\cite{Haa}, Peller \cite{Pe1} and Spronk \cite[Section 3.2]{Spronk}.
Part (2) is a folklore result which follows, for example,
from \cite[Exercice 8.7]{Pau} and the proof of \cite[Theorem 1.7]{DL}.

In the next statement, 
$(\,\cdotp\,\big\vert\,\,\cdotp)_{\footnotesize{\H}}$ stands for the inner
product on some Hilbert space $\H$.

\begin{theorem}\label{2Description}
Let $\phi\in L^\infty(\Omega^2)$ and let $C\geq 0$.
\begin{itemize}
\item [(1)] The function $\phi$ is a bounded Schur multiplier and 
$\norm{T_\phi\colon B(L^2(\Omega)\to B(L^2(\Omega))}\leq C$ if and only if there
exist a Hilbert space $\H$ and two functions 
$\alpha,\beta\in L^\infty(\Omega;\H)$ such that 
$\norm{\alpha}_\infty\norm{\beta}_\infty\leq C$ and 
$$
\phi(s,t) =\bigl(\alpha(t)\,\big\vert\,\beta(s)\bigr)_{\footnotesize{\H}}
\quad \hbox{a.e.-}(s,t)\in\Omega^2.
$$
\item [(2)] The function $\phi$ defines a bounded, positive Schur multiplier 
$T_\phi$ with
$\norm{T_\phi}\leq C$ if and only if there
exist a Hilbert space $\H$ and 
$\alpha\in L^\infty(\Omega;\H)$ such that $\norm{\alpha}_\infty^2\leq C$ and 
\begin{equation}\label{2Description2}
\phi(s,t) =\bigl(\alpha(t)\,\big\vert\,\alpha(s)\bigr)_{\footnotesize{\H}}
\quad \hbox{a.e.-}(s,t)\in\Omega^2.
\end{equation}
\end{itemize}
\end{theorem}

\medskip
Recall that $(\Omega,\mu)$ is called separable if $L^2(\Omega)$ is separable. For any $1\leq p<\infty$, this
is equivalent to the separability of $L^p(\Omega)$.

\begin{remark}
Assume that $(\Omega,\mu)$ is separable, let $\phi\in L^\infty(\Omega^2)$ and assume
that $\phi$ satisfies the assertion (ii) of
Theorem \ref{1Main}. Since $(N,\tau_N)$ is normalized, we have a contractive embedding 
$\kappa\colon N\to L^2(N,\tau_N)$. Let $\alpha :=\kappa\circ d \in L^\infty_\sigma(\Omega;L^2(N,\tau_N))$.
Then $\norm{\alpha}_\infty\leq 1$. 
According to (\ref{3RN}), $\alpha$ actually belongs to 
$L^\infty(\Omega;L^2(N,\tau_N))$. Moreover for almost every $(s,t)\in\Omega^2$, we have
$$
\tau_N\bigl(d(s)^*d(t)\bigr)= \bigl(\alpha(t)\,\big\vert\,\alpha(s)\bigr)_{L^2(N,\tau_N)}.
$$
Hence $\phi$ satisfies (\ref{2Description2}), by (\ref{1TraceForm}). Thus, 
(\ref{1TraceForm}) should be regarded as a strengthening
of the factorization (\ref{2Description2}). In this respect, Theorem \ref{1Main} says that this strengthened factorization property
characterizes Schur multipliers with a separable absolute dilation.
\end{remark}

\section{Preparatory results}\label{Prep}

Throughout we fix a $\sigma$-finite measure space $(\Omega,\mu)$.
For any $\theta\in L^\infty(\Omega)$, let $R_\theta\in B(L^2(\Omega))$ be the
pointwise multiplication operator defined by 
$R_\theta(h)=\theta h$ for all $h\in L^2(\Omega)$. Since $\theta\mapsto R_\theta$ is
a $w^*$-continuous 1-1 $*$-homomorphism, we may identify 
$L^\infty(\Omega)$ with $\{R_\theta\, :\,\theta \in L^\infty(\Omega)\}$.
Using this identification, we will henceforth regard $L^\infty(\Omega)$
as a von Neumann subalgebra of $B(L^2(\Omega))$, and simply write
\begin{equation}\label{4Mult}
L^\infty(\Omega)\subset B(L^2(\Omega)).
\end{equation}
We recall that $L^\infty(\Omega)$ is equal to its own commutant. That is, 
an operator $T\in B(L^2(\Omega))$ is of the form $R_{\theta_0}$ for some $\theta_0\in L^\infty(\Omega)$
if and only if $TR_\theta=R_\theta T$ for all $\theta\in L^\infty(\Omega)$,
that is, $T(\theta h)=\theta  T(h)$ for all $\theta\in L^\infty(\Omega)$ and all $h\in L^2(\Omega)$
(see \cite[Chapter IX, Theorem 6.6]{Conway2}).

For any measurable set $E\subset \Omega$, we let 
$\chi_E$ denote the indicator function of $E$. Note that 
$\chi_E$ belongs to $L^2(\Omega)$ if and only if
$E$ has finite measure.

\begin{lemma}\label{4MultOp}
Let $T\in B(L^2(\Omega))$. Then $T$ belongs to $L^\infty(\Omega)$ (in the
above sense) if and only if 
$T(\chi_E)=\chi_E T(\chi_E)$ for all measurable sets $E\subset \Omega$
with finite measure.
\end{lemma}

\begin{proof}
The `only if' part is obvious. To prove the `if part', let
$\A$ be the collection of all measurable sets $E\subset \Omega$
with finite measure.
We assume that $T(\chi_E)=\chi_E T(\chi_E)$
for all $E\in\A$. This readily implies that $T(\chi_E)=\chi_{E'} T(\chi_E)$
for all $E,E'\in\A$ with $E\subset E'$.

Let $E,F\in\A$. Since $\chi_E\chi_F=\chi_{E\cap F}$ and $E\cap F\subset E$,
it follows from above that we have 
$\chi_E T(\chi_E\chi_F)=T(\chi_E\chi_F)$.
Let $E^c$ be the complement of $E$. Since $\chi_{E^c}\chi_F=
\chi_{E^c\cap F}$, we have $T(\chi_{E^c}\chi_F)=\chi_{E^c\cap F} T(\chi_{E^c}\chi_F)$,
hence $\chi_E T(\chi_{E^c}\chi_F)=0$. Writing a decomposition $T(\chi_F)=T(\chi_E\chi_F)+
T(\chi_{E^c}\chi_F)$, we deduce that 
\begin{equation}\label{4Mod}
T(\chi_E\chi_F)=\chi_E T(\chi_F).
\end{equation}

Let
$$
\E={\rm Span}\bigl\{\chi_E\, :\, E\in\A\bigr\}.
$$
By linearity, the identity (\ref{4Mod})
implies that $T(\theta  h)=\theta T(h)$ for all $\theta, h\in\E$.
Since $\E$ is dense in $L^2(\Omega)$, we deduce that $T(\theta h)=\theta T(h)$ 
for all $h\in L^2(\Omega)$ and all $\theta\in\E$. Further $\E$ is $w^*$-dense in 
$L^\infty(\Omega)$. If $(\theta_\iota)_{\iota}$ is a net of $\E$ converging to
some $\theta\in L^\infty(\Omega)$ in the $w^*$-topology, then for all 
$h\in L^2(\Omega)$, 
$\theta_\iota h\to \theta h$ weakly in $L^2(\Omega)$. We deduce that
$T(\theta_\iota h)\to T(\theta h)$ and $\theta_\iota T(h)\to \theta T(h)$  
weakly in $L^2(\Omega)$. Consequently, we have
$T(\theta  h)=\theta T(h) $
for all $h\in L^2(\Omega)$ and all $\theta\in L^\infty(\Omega)$.
According to the discussion after (\ref{4Mult}), $T$ is therefore 
a multiplication operator, that is, an element of $L^\infty(\Omega)$.
\end{proof}

In the sequel we write 
$L^\infty_s(\Omega)\overline{\otimes}L^\infty_t(\Omega)$ instead 
of $L^\infty(\Omega)\overline{\otimes}L^\infty(\Omega)$ to indicate
that the variable in the first factor is denoted by $s\in\Omega$
and the  variable in the second factor is denoted by $t\in\Omega$.
In accordance with this convention, we re-write (\ref{3Omega2}) as
\begin{equation}\label{3Omega3}
L^\infty(\Omega^2)\simeq L^\infty_s(\Omega)\overline{\otimes} L^\infty_t(\Omega).
\end{equation}
Let $M$ be a von Neumann algebra and let
$$
W=L^\infty(\Omega^2)\overline{\otimes}M =
L^\infty_s(\Omega)\overline{\otimes}L^\infty_t(\Omega)\overline{\otimes} M.
$$
We let $1_t$ (resp. $1_s$) denote the constant function $1$ in 
$L^\infty_t(\Omega)$ (resp. $L^\infty_s(\Omega)$).
Consider $D$
and $C$ in $L^\infty(\Omega)\overline{\otimes} M$. We let
\begin{equation}\label{Ref1}
1_s\otimes D_t\in W
\end{equation}
be obtained from $D$
by identifying $L^\infty(\Omega)\overline{\otimes} M$
with $1_s\,\overline{\otimes} L^\infty_t(\Omega)\overline{\otimes} M\subset W$
in the natural way. Likewise, we let
\begin{equation}\label{Ref2}
1_t\otimes C_s\in W
\end{equation}
be obtained from $C$
by identifying $L^\infty(\Omega)\overline{\otimes} M$
with $L^\infty_s(\Omega)\overline{\otimes} \,1_t\, \overline{\otimes}M \subset W$.

\begin{lemma}\label{3cd2} 
Assume that $M$ has a separable predual. Let $C,D$ as above and
let $c,d\in L^\infty_\sigma(\Omega;M)$ such that $c\sim C$ and $d\sim D$.
Then the function $c\odot d\colon\Omega^2\to M$ defined by 
\begin{equation}\label{7cd}
(c\odot d)(s,t)=c(s)d(t),\qquad (s,t)\in\Omega^2,
\end{equation}
is $w^*$-measurable, hence belongs to 
$L^\infty_\sigma(\Omega^2;M)$. Moreover we have
\begin{equation}\label{6cd}
c\odot d\sim (1_t\otimes C_s)(1_s\otimes D_t).
\end{equation}
\end{lemma}

\begin{proof} 
Let $\eta\in M_*$.
For any $\delta\in M$, let $\eta\delta\in M_*$ be defined by
$\langle a,\eta\delta\rangle= \langle \delta a, \eta\rangle$ for all $a\in M$.
This is well-defined since the product on $M$ is separately 
$w^*$-continuous. Similarly, $\delta\eta\in M_*$ is defined by
$\langle a,\delta\eta\rangle= \langle  a\delta, \eta\rangle$.

Writing
$\langle\delta,  d(t)\eta\rangle = \langle d(t),\eta\delta\rangle$ for any 
$\delta\in M$, we see that the function
$t\mapsto d(t)\eta$ is weakly measurable from $\Omega$ to $M_*$. 
Hence it is measurable, by the
separability assumption and \cite[Theorem II. 1.2]{DU}. 
The function
$t\mapsto d(t)\eta$ is therefore an 
almost everywhere limit of a sequence $(w_n)_{n\geq 1}$ of $L^\infty_t(\Omega)\otimes M_*$.
Since $c$ is $w^*$-measurable, the function $(s,t)\mapsto \langle c(s),w_n(t)\rangle$
is measurable for any $n\geq 1$.
We deduce that  $(s,t)\mapsto \langle c(s),d(t)\eta\rangle =  \langle c(s)d(t),\eta\rangle$
is measurable. This shows that $c\odot d$ is $w^*$-measurable. It is clear that
$c\odot d$ is essentially bounded. Thus, we
obtain that $c\odot d\in L^\infty_\sigma(\Omega^2;M)$.

Set  $\Delta =(1_t\otimes C_s)(1_s\otimes D_t)$ in this proof.  
We introduce
$$
m=\int_\Omega f(s) c(s)\,ds\, \in M,
$$
this integral being defined in  the
$w^*$-topology.
Consider the slice map  
$$
\widetilde{f} : =f\overline{\otimes} I_{L^\infty_t(\Omega)\overline{\otimes} M}
\colon W\longrightarrow L^\infty(\Omega)\overline{\otimes} M.
$$
For any $u,v\in L^\infty(\Omega)$ and any $a,b\in M$, we have
$(u\otimes 1\otimes a)(1\otimes v\otimes b)
=(u\otimes v\otimes ab)$ in $W$, hence
$$
\widetilde{f}\bigl((u\otimes 1\otimes a)(1\otimes v\otimes b)\bigr)
= \Bigl(\int_\Omega  f u \Bigr) v\otimes ab
=\Bigl(1\otimes \Bigl(\int_\Omega f(s)u(s)a\, ds\Bigr)\Bigr)
(v\otimes b).
$$
By linearity, this implies that if $C$ and $D$
belong to the algebraic tensor product
$L^\infty(\Omega)\otimes M$, we have
\begin{equation}\label{3Form}
\widetilde{f}(\Delta)
=(1\otimes m)D.
\end{equation}
If $(C_\iota)_\iota$ is a net of $ L^\infty(\Omega)\overline{\otimes} M$ converging to $C$ in
the $w^*$-topology of 
$L^\infty(\Omega)\overline{\otimes} M$ and if
$c_\iota\in L^\infty_\sigma(\Omega;M)$ are such that $c_\iota\sim C_\iota$,
then $\int_\Omega f(s) c_\iota(s)\,ds\,$ converges to $m$ in the $w^*$-topology of
$M$.
Since the product on von Neumann algebras is separately 
$w^*$-continuous and $\widetilde{f}$ is $w^*$-continuous, we deduce that the identity (\ref{3Form}) actually holds 
in the general case. 
Consequently,
$$
[f\otimes g\otimes\eta](\Delta) =
[g\otimes\eta]\bigl((1\otimes m)D\bigr).
$$

Since $d\sim D$, Lemma \ref{3x} and the `only if' part of Remark \ref{3Sakai2} ensure that 
$$
[g\otimes\eta]\bigl((1\otimes m)D\bigr)=\int_\Omega g(t) \langle md(t),\eta\rangle\,
dt.
$$
Writing 
$$
\langle md(t),\eta\rangle=\langle m,d(t)\eta\rangle
=\int_\Omega f(s)\langle c(s),d(t)\eta\rangle\,ds
=\int_\Omega f(s)\langle c(s)d(t), \eta\rangle\,ds,
$$
and using the measurability of $\langle c\odot d,\eta\rangle$, we deduce that 
$$
[f\otimes g\otimes\eta](\Delta) 
=\int_{\Omega^2} g(t)f(s)
\bigl\langle c(s)d(t),\eta\bigr\rangle\, dsdt.
$$
Applying the `if' part of Remark \ref{3Sakai2}, we deduce (\ref{6cd}).
\end{proof}

Let $H_1,H_2$ be two Hilbert spaces and let $H_1\stackrel{2}{\oplus} H_2$ denote their
Hilbertian direct sum. Let $B(H_1,H_2)$ denote the Banach space
of all bounded operators from $H_1$ into $H_2$. Any $L\in  B(H_1\stackrel{2}{\oplus} H_2)$
can be naturally written as a $2\times 2$ matrix
$$
L=\begin{pmatrix} L_{11} & L_{12}\\ L_{21} & L_{22}\end{pmatrix}
$$
with $L_{ij}\in B(H_j,H_i)$ for all $(i,j)\in\{1,2\}^2$. Thus, 
we may regard 
\begin{equation}\label{4embedding}
B(H_1,H_2)\subset B(H_1\stackrel{2}{\oplus} H_2)
\end{equation} 
as a closed subspace.
Note that $B(H_1,H_2)$ is actually a $w^*$-closed subspace of $B(H_1\stackrel{2}{\oplus} H_2)$,
hence a dual space. If $H_1,H_2$ are separable, then the predual of 
$B(H_1,H_2)$ is separable as well.

For any von Neumann algebra $M_0$, we let $M_0\overline{\otimes} B(H_1,H_2)$
denote the $w^*$-closure of the algebraic tensor product $M_0\otimes B(H_1,H_2)$ into 
$M_0\overline{\otimes} B(H_1\stackrel{2}{\oplus} H_2)$. For any 
$$
D\in M_0\overline{\otimes} B(H_1,H_2)\qquad\hbox{and}\qquad
C\in M_0\overline{\otimes} B(H_2,H_1),
$$
the product of $D$ and $C$ in 
the von Neumann algebra $M_0\overline{\otimes} B(H_1\stackrel{2}{\oplus} H_2)$
yields 
$$
CD\in M_0\overline{\otimes} B(H_1).
$$

Let $K$
be another Hilbert space and let $H=H_1\stackrel{2}{\oplus} H_2$. We 
note that the identification
$B(K\stackrel{2}{\otimes} H) \simeq  B(K)\overline{\otimes} B(H)$ reduces to
a $w^*$-homeomorphic and isometric identification
\begin{equation}\label{4Tensor}
B(K\stackrel{2}{\otimes} H_1, K\stackrel{2}{\otimes}  H_2) \simeq  B(K)\overline{\otimes} B(H_1,H_2).
\end{equation}

Part (2) of Lemma \ref{3Sakai} extends to the case when $M$ is replaced by $B(H_1,H_2)$. 
More precisely, if $H_1,H_2$ are separable, then the natural embedding
of $L^\infty(\Omega)\otimes B(H_1,H_2)$ into 
$L^\infty_\sigma(\Omega;B(H_1,H_2))$
extends to a $w^*$-homeomorphic and isometric identification
\begin{equation}\label{H12}
L^\infty_\sigma(\Omega;B(H_1,H_2))\simeq L^\infty(\Omega)\overline{\otimes}B(H_1,H_2).
\end{equation}
To prove it, apply Lemma \ref{3Sakai} with $M= B(H_1\stackrel{2}{\oplus} H_2)$ and use the embedding (\ref{4embedding}).

For any 
$D\in L^\infty(\Omega)\overline{\otimes} B(H_1,H_2)$, we may associate, similarly to what is done before Lemma \ref{3cd2},
$$
1_s\otimes D_t \in 
L^\infty_s(\Omega)\overline{\otimes}L^\infty_t(\Omega)\overline{\otimes} B(H_1,H_2)
$$
by identifying $L^\infty(\Omega)\overline{\otimes} B(H_1,H_2)$
with $1_s\,\overline{\otimes} L^\infty_t(\Omega)\overline{\otimes} B(H_1,H_2)$.

Similarly, for any  $C\in L^\infty(\Omega)\overline{\otimes} B(H_2,H_1)$, we may associate
$$
1_t\otimes C_s\in 
L^\infty_s(\Omega)\overline{\otimes}L^\infty_t(\Omega)\overline{\otimes} B(H_2,H_1).
$$
by identifying $L^\infty(\Omega)\overline{\otimes} B(H_2,H_1)$
with $L^\infty_s(\Omega)\overline{\otimes} 1_t\overline{\otimes} B(H_2,H_1)$.

\begin{lemma}\label{4N}
Let $D,C$ as above and recall from (\ref{4Mult}) that they may be regarded as elements 
of $B(L^2(\Omega))\overline{\otimes} B(H_1,H_2)$ and 
$B(L^2(\Omega))\overline{\otimes} B(H_2,H_1)$, respectively. 
Let $N_1\subset B(H_1)$ be a von Neumann subalgebra.
The following assertions are equivalent.
\begin{itemize}
\item [(i)]
For all $u,v\in L^2(\Omega)$, $C((u\otimes v)\otimes I_{H_2})D$ belongs to
$B(L^2(\Omega))\overline{\otimes} N_1$. (Here $u\otimes v$ is the element
of $B(L^2(\Omega))$ defined by (\ref{uv}).)
\item [(ii)] The product $(1_t\otimes C_s)(1_s\otimes D_t)$ belongs
to $L^\infty_s(\Omega)\overline{\otimes}L^\infty_t(\Omega)\overline{\otimes} 
N_1$.
\end{itemize}
\end{lemma}

\begin{proof}
Let $u,v,a,b\in L^2(\Omega)$. Following (\ref{uv}),
we consider $a\otimes b$ as an element
of $S^1(L^2(\Omega))$ and we let $(a\otimes b)\overline{\otimes} I$
be the associated slice map 
$B(L^2(\Omega))\overline{\otimes} B(H_1)\to
B(H_1)$. Likewise, we regard $av\otimes bu$ as an element
of $L^1_s(\Omega)\otimes L^1_t(\Omega)$ and we let
$(av\otimes bu)\overline{\otimes} I$ be the associated slice map
 $L^\infty_s(\Omega)\overline{\otimes}L^\infty_t(\Omega)\overline{\otimes} B(H_1)
\to B(H_1)$. We claim that 
\begin{equation}\label{4Magic}
\bigl[(a\otimes b)\overline{\otimes} I\bigr]\bigl(C((u\otimes v)\otimes I_{H_2})D\bigr)
= \bigl[(av\otimes bu)\overline{\otimes} I\bigr]\bigl((1_t\otimes C_s)(1_s\otimes D_t)\bigr).
\end{equation}
To prove this, assume first that
$D\in L^\infty(\Omega)\otimes B(H_1,H_2)$ and 
$C\in  L^\infty(\Omega)\otimes B(H_2,H_1)$. Thus,
$$
D=\sum_j  f_j\otimes T_j
\qquad\hbox{and}\qquad
C=\sum_i g_i\otimes S_i
$$
for some finite families $(f_j)_j$ in $L^\infty(\Omega)$, 
$(T_j)_j$ in $B(H_1,H_2)$, $(g_i)_i$ in $L^\infty(\Omega)$ and
$(S_i)_i$ in $B(H_2,H_1)$.

Recall from (\ref{4Mult}) that 
each $f_j$ (resp. $g_i$) is identified with the pointwise multiplication
by $f_j$ (resp. $g_i$). Then an elementary calculation yields 
$g_i(u\otimes v)f_j= (uf_j)\otimes(vg_i)$ in $B(L^2(\Omega))$. Consequently,
$$
C((u\otimes v)\otimes I_{H_2})D = \sum_{i,j}
(uf_j)\otimes(vg_i)\otimes S_iT_j.
$$
Hence, the left-hand side of (\ref{4Magic}) is
\begin{align*}
\bigl[(a\otimes b)\overline{\otimes} I\bigr]\bigl(C((u\otimes v)\otimes I_{H_2})D\bigr)
&=\sum_{i,j}
{\rm tr}\bigl((a\otimes b)[(uf_j)\otimes(vg_i)]\bigr) S_iT_j\\
& = \sum_{i,j} \Bigl(\int_\Omega avg_i\Bigr)
\Bigl(\int_\Omega buf_j\Bigr)S_iT_j.
\end{align*}
To compute the right-hand side of (\ref{4Magic}), write
$$
(1_t\otimes C_s)(1_s\otimes D_t)
=\Bigl(\sum_i  g_i\otimes 1\otimes S_i\Bigr)
\Bigl(\sum_j 1\otimes f_j\otimes T_j\Bigr)
=\sum_{i,j} g_i\otimes f_j\otimes S_iT_j.
$$
This implies
$$
\bigl[(av\otimes bu)\overline{\otimes}  I\bigr]\bigl((1_t\otimes C_s)(1_s\otimes D_t)\bigr)
 = \sum_{i,j} \Bigl(\int_\Omega avg_i\Bigr)
\Bigl(\int_\Omega buf_j\Bigr)S_iT_j,
$$
and (\ref{4Magic}) follows. Now since the product on von Neumann algebras is 
separately $w^*$-continuous, the above special case implies that (\ref{4Magic})
holds true in the general case.

We now prove the equivalence. Assume (i) and let $h,k\in L^1(\Omega)$.
There exist $u,v,a,b\in L^2(\Omega)$ such that $h=av$ and $k=bu$. 
Since $C((u\otimes v)\otimes I_{H_2})D$ belongs to
$B(L^2(\Omega))\overline{\otimes} N_1$, it follows from 
(\ref{4Magic}) that 
$\bigl[(h\otimes k)\overline{\otimes} I\bigr]\bigl((1_t\otimes C_s)(1_s\otimes D_t)\bigr)$
belongs to $N_1$. Since $L^1(\Omega)_s\otimes L^1_t(\Omega)$ is
dense in $(L^\infty_s(\Omega)\overline{\otimes}L^\infty_t(\Omega))_*$,
we deduce that $(1_t\otimes C_s)(1_s\otimes D_t)\in
L^\infty_s(\Omega)\overline{\otimes}L^\infty_t(\Omega)\overline{\otimes} N_1$
by applying the slice map property (see Remark \ref{3Slice}). 

Conversely, assume (ii) and fix $u,v\in L^2(\Omega)$.
Since  $(1_t\otimes C_s)(1_s\otimes D_t)\in
L^\infty_s(\Omega)\overline{\otimes}L^\infty_t(\Omega)\overline{\otimes} N_1$,
it follows from 
(\ref{4Magic}) that
$\bigl[(a\otimes b)\overline{\otimes} I\bigr]\bigl(C((u\otimes v)\otimes I_{H_2})D\bigr)$
belongs to $N_1$ for all $a,b\in L^2(\Omega)$. Since 
$L^2(\Omega)\otimes L^2(\Omega)$ is dense in $B(L^2(\Omega))_*$,
we deduce that $C((u\otimes v)\otimes I_{H_2})D\in
B(L^2(\Omega))\overline{\otimes} N_1$
by applying again the slice map property.
\end{proof}

\section{Proof of Theorem \ref{1Main}}\label{SProof}

\subsection{Proof of ``$(i)\,\Rightarrow\,(ii)$"}\label{61}
We let $\phi\in L^\infty(\Omega^2)$ and we assume that $\phi$
satisfies the assertion (i) of Theorem \ref{1Main}. 

Applying (\ref{2Dil1}) with $k=1$, 
we obtain a tracial von Neumann algebra
$(\M,\tau_{\mathcal M})$ with a separable predual, a trace preserving unital 
$w^*$-continuous $*$-homomorphism $J\colon B(L^2(\Omega))\to\M$
and a trace preserving $*$-automorphism 
$U\colon\M\to\M$ such that
\begin{equation}\label{4k=1}
T_\phi=J_1^*UJ.
\end{equation}
We fix a Hilbertian basis $(e_i)_{i\in I}$ of $L^2(\Omega)$
and we define
$E_{i j}=e_j\otimes e_i\in B(L^2(\Omega))$ for all $(i,j)\in I^2$.
Since $J$ is a unital $w^*$-continuous $*$-homomorphism, 
$\{J(E_{ij})\, :\, (i,j)\in I^2\}$ is a system of matrix
units of $\M$ (in the sense of \cite[Definition IV.1.7]{Tak}). Fix
$k_0\in I$ and set $q=J(E_{k_0k_0})$. This is a projection 
of $\M$. According to \cite[Proposition IV.1.8]{Tak} and its proof,
$$
m_{ij}:=J(E_{k_0 i}) m J(E_{jk_0})
$$
belongs to $q\M q$ for all $m\in\M$
and all $(i,j)\in I^2$, and the mapping
$$
\rho\colon \M\longrightarrow B(L^2(\Omega))\overline{\otimes}(q\M q),
\qquad \rho(m) = \sum_{(i,j)\in I^2} E_{ij}\otimes m_{ij},
$$
is a $*$-isomorphism. (Here the summation is taken in the $w^*$-topology of
$\M$.) The restriction of $\tau_{\footnotesize{\M}}$ to $q\M q$ 
is semi-finite, let us call it $\tau_1$. Equip $B(L^2(\Omega))\overline{\otimes}(q\M q)$
with ${\rm tr}\overline{\otimes}\tau_1$, see Remark \ref{TP-trace}.
Then $\rho$ is trace preserving. To check this, observe that 
$J(E_{ik_0})J(E_{k_0 i}) =J(E_{ii})$ for all $i\in I$ and that
$\sum_{i\in I} J(E_{ii}) = 1_{\mathcal{M}}$.
Hence for any $m\in\M^+$, each
$m_{ii}$ belongs to $(q\M q)^+$ and we have
\begin{align*}
({\rm tr}\overline{\otimes}\tau_1)\bigl(\rho(m)\bigr)
& = \sum_{i\in I}\tau_{\mathcal{M}}\bigl( J(E_{k_0 i})m J(E_{ik_0})\bigr)\\
& = \sum_{i\in I} \tau_{\mathcal{M}}\bigl(J(E_{i i}) m \bigr)\\
& = \tau_{\mathcal{M}} \Bigl(\Bigl(\sum_{i\in I} J(E_{ii})\Bigr) m \Bigr) = \tau_{\mathcal{M}}(m).
\end{align*}
It is plain that $\rho\bigl(J(E_{ij})\bigr)= E_{ij}\otimes q$ for all $(i,j)\in I^2$.
Since $J$ is $w^*$-continuous, this implies that 
$\rho\circ J (z) = z \otimes q,$ for all $z\in B(L^2(\Omega))$.
Since both $J$ and $\rho$ are trace preserving, this implies that $\tau_1(q)=1$.

With $N_1=q\M q$, it follows from above that
changing $J$ into $\rho\circ J$ and changing $U$ into $\rho\circ U\circ\rho^{-1}$, 
we may assume that
$$
\M=B(L^2(\Omega))\overline{\otimes} N_1, 
$$
for some normalized tracial von Neumann algebra
$(N_1,\tau_1)$, and
\begin{equation}\label{5J}
J(z) =z\otimes 1_{N_1},\qquad z
\in B(L^2(\Omega)).
\end{equation}
Moreover $N_1$ has a separable predual.

We fix a Hilbert space 
$H_1$ such that $N_1\subset B(H_1)$ as a von Neumann subalgebra. 
We may assume that $H_1$ is separable, see Remark \ref{2Dil2}, (5).
By (\ref{Hilbert-TP}), we have 
$B(L^2(\Omega))\overline{\otimes}B(H_1)\simeq B(L^2(\Omega;H_1))$.
Hence $\M\subset  B(L^2(\Omega;H_1))$ as a von Neumann subalgebra. 
We use again \cite[Proposition IV.1.8]{Tak} and its proof as follows. 
We note that since $U$ is a unital $w^*$-continuous $*$-homomorphism, 
$\{U(E_{ij}\otimes 1_{N_1})\, :\, (i,j)\in I^2\}$ is a system of matrix
units of $\M$. Then we set
$$
q_0 =U(E_{k_0k_0}\otimes 1_{N_1}),\quad
H_2= q_0\bigl(L^2(\Omega;H_1)\bigr)
\quad \hbox{and} \quad
N_2 =q_0\M q_0,
$$
and we equip the latter von Neumann algebra with the restriction of the trace of $\M$.
We let $\tau_2$ denote this trace on $N_2$.
Then arguing as above, 
we find a trace preserving $*$-isomorphism
\begin{equation}\label{4pi0}
\pi \colon  B(L^2(\Omega))\overline{\otimes} N_1
\longrightarrow B(L^2(\Omega))\overline{\otimes} N_2
\end{equation}
such that $\pi\circ U(z\otimes 1_{N_1}) = z\otimes 1_{N_2}$
for all $z\in B(L^2(\Omega))$. We deduce that $\tau_2$
is normalized.
Moreover, $N_2\subset B(H_2)$ as a von Neumann subalgebra, and
$H_2$ is separable.

Regard $B(L^2(\Omega))\overline{\otimes} N_1\subset B(L^2(\Omega;H_1))$
as above and 
$B(L^2(\Omega))\overline{\otimes} N_2\subset B(L^2(\Omega;H_2))$
similarly.
Then the argument in the proof
of \cite[Proposition IV.1.8]{Tak} actually shows the existence
of a unitary operator
$$
D\colon L^2(\Omega;H_1)\longrightarrow L^2(\Omega; H_2)
$$
such that 
\begin{equation}\label{4pi}
\pi(R) = DRD^*,\qquad R\in B(L^2(\Omega))\overline{\otimes} N_1.
\end{equation}
This implies that
\begin{equation}\label{4U}
U(z\otimes 1_{N_1}) = D^*(z\otimes 1_{N_2})D,\qquad 
z\in B(L^2(\Omega)).
\end{equation}

We shall now use properties of Schur multipliers.
Note that $1_{N_1}=I_{H_1} \in L^1(N_1,\tau_1)$ and recall the embedding $S^1(L^2(\Omega))\otimes L^1(N_1,\tau_1)
\subset \M_*$ from (\ref{3Predual}).
The relation (\ref{4k=1}) and the identity (\ref{5J}) imply that for all $u,v,a,b\in L^2(\Omega)$, we have
\begin{equation}\label{4Exp}
\bigl\langle U((u\otimes v)\otimes I_{H_1}),(a\otimes b)\otimes 
I_{H_1}\bigr\rangle_{{\mathcal M},
{\mathcal M}_*}\,=\,
{\rm tr}\Bigl(\bigl[T_\phi(u\otimes v)\bigr]\bigl[a\otimes b\bigr]\Bigr).
\end{equation}
In this formula, $u\otimes v$ (and similarly $a\otimes b$)
is the rank one operator on $L^2(\Omega)$
defined by (\ref{uv}).
Note that the right-hand side of (\ref{4Exp}) can be expressed by (\ref{2Trace}).

According to (\ref{4Tensor}), we regard $D$ as an element of 
$B(L^2(\Omega))\overline{\otimes} B(H_1,H_2)$.   

\begin{lemma}\label{4D}
Under the representation  (\ref{4Mult}), we have
$$
D\in L^\infty(\Omega)\overline{\otimes} B(H_1,H_2).
$$
\end{lemma}

\begin{proof}
Let $E,F\subset \Omega$ be measurable sets of finite measure.
Consider 
$$
V_{E,F} = (\chi_F\otimes \chi_F\otimes I_{H_2})D(\chi_E\otimes \chi_E\otimes I_{H_1})
\,\in B(L^2(\Omega))\overline{\otimes} B(H_1,H_2).
$$
Since $(\chi_F\otimes \chi_F\otimes I_{H_2})$ is a self-adjoint projection, 
we have 
$$
V_{E,F}^*V_{E,F} = (\chi_E\otimes \chi_E\otimes I_{H_1})D^*
(\chi_F\otimes \chi_F\otimes I_{H_2})D(\chi_E\otimes \chi_E\otimes I_{H_1}).
$$
We regard this product as an element of $B(L^2(\Omega))\overline{\otimes}B(H_1)$. According to  (\ref{4U}),
$$
D^*(\chi_F\otimes \chi_F\otimes I_{H_2})D = U(\chi_F\otimes \chi_F\otimes I_{H_1})
$$
belongs to 
$B(L^2(\Omega))\overline{\otimes}N_1$. Hence
$V_{E,F}^*V_{E,F}$ actually belongs to $B(L^2(\Omega))\overline{\otimes}N_1$.
Moreover using (\ref{4Exp}) and (\ref{2Trace}), we have
\begin{align*}
({\rm tr}\overline{\otimes}\tau_1)\bigl(V_{E,F}^*V_{E,F}\bigr)
& = ({\rm tr}\overline{\otimes}\tau_1)\Bigl(\bigl[U(\chi_F\otimes \chi_F\otimes I_{H_1})\bigr]
\bigl[\chi_E\otimes \chi_E\otimes I_{H_1}\bigr]\Bigr)\\
&=\bigl\langle U(\chi_F\otimes \chi_F\otimes I_{H_1}), \chi_E\otimes \chi_E\otimes 
I_{H_1}\bigr\rangle_{{\mathcal M},
{\mathcal M}_*}\\
& = {\rm tr}\Bigl(\bigl[T_\phi(\chi_F\otimes \chi_F)\bigr]
\bigl[\chi_E\otimes \chi_E\bigr]\Bigr)\\
& = \int_{\Omega^2}\phi(s,t)\chi_F(t)\chi_F(s)\chi_E(s)\chi_E(t)\, dtds.
\end{align*}
Consequently, 
$({\rm tr}\overline{\otimes}\tau_1)\bigl(V_{E,F}^*V_{E,F}\bigr)=0$ if 
$E$ and $F$ are disjoint. Hence 
$V_{E,F}=0$ if $E$ and $F$ are disjoint.

For any $\eta\in B(H_1,H_2)_*$, set
$D_\eta = (I_{B(L^2)}\overline{\otimes}\eta)(D)\in B(L^2(\Omega))$. 
For all $z\in B(L^2(\Omega))$ and
$w\in B(H_1,H_2)$, we have
\begin{align*}
(I_{B(L^2)}\otimes\eta)\bigl[
(\chi_F\otimes \chi_F\otimes I_{H_2})(z\otimes w)
&(\chi_E\otimes \chi_E\otimes I_{H_1})\bigr]\\
&= 
(I_{B(L^2)}\otimes\eta)\bigl[
(\chi_F\otimes \chi_F)z(\chi_E\otimes \chi_E)\otimes w\bigr]\\
&=\eta(w)(\chi_F\otimes \chi_F)z(\chi_E\otimes \chi_E)\\
&= (\chi_F\otimes \chi_F)\bigl[(I_{B(L^2)}\otimes\eta)(z\otimes w)\bigr]
(\chi_E\otimes \chi_E).
\end{align*}
Approximating $D$ by elements of the algebraic 
tensor product $B(L^2(\Omega))\otimes B(H_1,H_2)$, we deduce  that  
$$
(I_{B(L^2)}\overline{\otimes}\eta)(V_{E,F}) = 
(\chi_F\otimes \chi_F)D_\eta
(\chi_E\otimes \chi_E).
$$
Hence 
$$
(\chi_F\otimes \chi_F)D_\eta(\chi_E\otimes \chi_E)=0
$$
if $E$ and $F$ are disjoint. Since
$$
(\chi_F\otimes \chi_F)D_\eta(\chi_E\otimes \chi_E) =
\Bigl(\int_\Omega D_\eta(\chi_E)\chi_F\Bigr)\chi_E\otimes\chi_F,
$$
this implies that 
$$
\int_\Omega D_\eta(\chi_E)\chi_F\,=0
$$
if $E$ and $F$ are disjoint. 
We deduce that $D_\eta(\chi_E)$ has support in $E$ whenever 
$E\subset\Omega$ is a measurable set of finite measure.
Applying Lemma \ref{4MultOp}, we obtain that
$D_\eta\in L^\infty(\Omega)$ for all $\eta\in B(H_1,H_2)_*$.
The result therefore follows from the slice map property (see Remark \ref{3Slice}).
\end{proof}

\begin{lemma}\label{4Ess}
The product $(1_t\otimes D_s^*)(1_s\otimes D_t)$
belongs to $L^\infty_s(\Omega)\overline{\otimes}L^\infty_t(\Omega)\overline{\otimes} N_1$.
\end{lemma}

\begin{proof}
For all $u,v\in L^2(\Omega)$, we have
$D^*((u\otimes v)\otimes I_{H_2})D=\pi^{-1}((u\otimes v)\otimes I_{H_2})$,
by (\ref{4pi}), and the latter belongs to $B(L^2(\Omega))\overline{\otimes} N_1$.
The result therefore follows from  Lemma 
\ref{4N}.
\end{proof}

\begin{lemma}\label{4Trace}
In the space $L^\infty(\Omega^2)$, we have
\begin{equation}\label{4Trace1}
\phi= \bigl[I_{L^\infty(\Omega^2)}\overline{\otimes}\tau_1\bigr]
\bigl((1_t\otimes D_s^*)(1_s\otimes D_t)\bigr).
\end{equation}
\end{lemma}

\begin{proof} 
Let $\widetilde{\phi}\in L^\infty(\Omega^2)$ be equal to the right-hand
side of (\ref{4Trace1}).
Let $a,b,u,v\in L^2(\Omega)$. By (\ref{2Trace}), 
(\ref{4Exp}), (\ref{4U}) and (\ref{4Magic}), we have
\begin{align*}
\int_{\Omega^2} \phi(s,t)u(t)v(s)a(s)b(t)\, dtds &
= \bigl\langle U((u\otimes v)\otimes I_{H_1}), (a\otimes b)\otimes 
I_{H_1}\bigr\rangle\\
& =  \bigl\langle D^*((u\otimes v)\otimes 1_{N_2})D, 
(a\otimes b)\otimes 1_{N_1} \bigr\rangle\\
&  =\tau_1\Bigl(\bigl[(a\otimes b)\overline{\otimes} I_{N_1}\bigr]\bigl(D^*((u\otimes v)\otimes 1_{N_2})D\bigr)\Bigr)\\
& = \tau_1\Bigl(\bigl[(av\otimes bu)\overline{\otimes} I_{N_1}
\bigr]\bigl((1_t\otimes D^*_s)(1_s\otimes D_t)\bigr)\Bigr)\\
&= \bigl[av\otimes bu \otimes\tau_1\bigr]\bigl((1_t\otimes D^*_s)(1_s\otimes D_t)\bigr)\\
&=\int_{\Omega^2} \widetilde{\phi}(s,t)  a(s)v(s)b(t)u(t)\, dtds.
\end{align*}
Here $I_{N_1}$ denotes the identity operator on $N_1$.
Consequently,
$$
\int_{\Omega^2} \phi(s,t)  f(s)g(t)\, dtds = 
\int_{\Omega^2} \widetilde{\phi}(s,t)
f(s)g(t)\, dtds 
$$
for all $f,g\in L^1(\Omega)$. The result follows.
\end{proof}

We can now prove the assertion (ii) of Theorem \ref{1Main}.
Recall that $H_1$ and $H_2$ are separable.
Let $d\in L^\infty_\sigma(\Omega; B(H_1,H_2))$ be associated with 
$D$ in the identification (\ref{H12}). By Lemmas \ref{3cd2} and \ref{4Ess}, we have 
$d^*\odot d\sim (1_t\otimes D_s^*)(1_s\otimes D_t)$ and
\begin{equation}\label{d*d}
d^*\odot d\in L^\infty_\sigma(\Omega;N_1).
\end{equation}
Further for any $f,g\in L^1(\Omega)$, it follows from (\ref{4Trace1}) and Remark \ref{3Sakai2} that
\begin{align*}
\int_{\Omega^2} f(s)g(t)\phi(s,t)\, dtds\,&= [f\otimes g\otimes\tau_1]\bigl(
(1_t\otimes D_s^*)(1_s\otimes D_t)\bigr)\\
&=\int_{\Omega^2} f(s)g(t)\tau_1\bigl( d(s)^*d(t)\bigr)\, dtds.
\end{align*}
This implies that
\begin{equation}\label{6Factor}
\phi(s,t) =\tau_{1}\bigl(d(s)^*d(t)\bigr), 
\qquad{\rm for\  a.e.}\ (s,t)\in\Omega^2.
\end{equation}.

Recall that as an operator from $L^2(\Omega;H_1)$ to  $L^2(\Omega;H_2)$,
$D$ is a unitary. Hence we both have  $D^*D=1_{L^\infty}\otimes I_{H_1}$ and
$DD^*=1_{L^\infty}\otimes I_{H_2}$. By Lemma \ref{3Sakai}, this implies
that $d(s)^*d(s)=1_{H_1}$ and $d(s)d(s)^*=I_{H_2}$ for almost
every $s\in\Omega$.
Changing the values of $d$ on a null subset of $\Omega$, we can henceforth
assume that for all $s\in\Omega$, $d(s)\colon H_1\to H_2$ is a unitary.

Consider the pre-annihilator $(N_{1})_\perp\subset B(H_1)_*$.
It follows from (\ref{d*d}) that for all $f,g\in L^1(\Omega)$ and for all 
$\eta\in (N_{1})_\perp$, we have
$$
\int_{\Omega^2} f(s)g(t)\langle d(s)^* d(t),\eta\rangle\, dtds\,=0.
$$
Since $(\Omega,\mu)$ is separable, $L^1(\Omega)$ is separable. 
Also, $B(H_1)_*$ is separable, hence $(N_{1})_\perp$ is separable.
Let $(g_k)_{k\geq 1}$ and $(\eta_n)_{n\geq 1}$ be dense sequences
of $L^1(\Omega)$ and $(N_{1})_\perp$, respectively. 
It follows from above that for all $n,m\geq 1$, 
$\int_\Omega g_k(t)\langle d(s)^* d(t),\eta_n\rangle\, dt=0$ for almost every 
$s\in\Omega$. Consequently, there exists $s_0\in\Omega$
such that 
\begin{equation}\label{Null}
\int_\Omega  g_k(t)\langle d(s_0)^* d(t),\eta_n\rangle\, dt=0
\end{equation}
for all $n,m\geq 1$.

Define $\widetilde{d}\colon\Omega\to B(H_1)$ by $\widetilde{d}(t)=d(s_0)^* d(t)$.
Clearly $\widetilde{d}$ is $w^*$-measurable and essentially bounded. Since
$d(s_0)$ is a unitary, we have $\widetilde{d}(t)^*\widetilde{d}(t)=d(t)^*d(t)=I_{H_1}$
for almost every $t\in\Omega$. 
Hence as an element of
$L^\infty_\sigma(\Omega, B(H_1))$,  $\widetilde{d}$ is a unitary.
According to (\ref{Null}), $\widetilde{d}$ belongs to
the annihilator of 
$$
{\rm Span}\{g_k\otimes \eta_n\, :\, k,n\geq 1\}\subset L^1(\Omega;B(H_1)_*).
$$
Since $\{g_k\otimes \eta_n\, :\, k,n\geq 1\}$ is dense in $L^1(\Omega;(N_{1})_\perp)$, this implies that 
$\widetilde{d}$ belongs to $L^1(\Omega;(N_{1})_\perp)^\perp$. By (\ref{3Inclusion2}), this means
that $\widetilde{d}\in L^\infty_\sigma(\Omega, N_1)$.

Finally, using again the fact that $d(s_0)$ is a unitary, we deduce from (\ref{6Factor}) that
$\phi(s,t) =\tau_{1}\bigl(\widetilde{d}(s)^*\widetilde{d}(t)\bigr)$ for almost every 
$(s,t)\in\Omega^2$. Hence the assertion (ii) of Theorem \ref{1Main} is satisfied, with $N=N_1$ 
and $\widetilde{d}$ instead of $d$.

\subsection{Proof of ``$(ii)\,\Rightarrow\,(i)$"}\label{62}
Assume that $\phi\in L^\infty(\Omega^2)$ satisfies 
the assertion (ii) of Theorem \ref{1Main}, for some normalized tracial von Neumann 
algebra $(N,\tau_N)$ with a separable predual and
some unitary $d\in L^\infty_\sigma(\Omega;N)$.

Since  $\tau_N\colon N\to\Cdb$ is a faithful state, we may define
$$
N^\infty = \mathop{\overline{\otimes}}_{\mathbb Z} N,
$$
the infinite tensor product of $(N,\tau_N)$ over the index set $\Zdb$, as
considered in \cite[Definition XIV.1.6]{Tak3}. The associated 
state $\tau_\infty := \mathop{\overline{\otimes}}_{\mathbb Z}\tau_N$ on $N^\infty$
is a normal faithful tracial state. Thus, $(N^\infty,\tau_\infty)$ is a 
normalized tracial von Neumann algebra.
We let
$$
\M=B(L^2(\Omega))\overline{\otimes} N^\infty,
$$
equipped with its natural trace (see Remark \ref{TP-trace}). The von Neumann
algebra $N^\infty$ has a separable predual, hence $\M$ has a separable predual.

Let $1_\infty$
denote the unit of $N^\infty$ and let $J\colon B(L^2(\Omega))\to \M$ be
the $w^*$-continuous trace preserving unital $*$-homomorphism defined by
$$
J(z)= z\otimes 1_\infty,\qquad z\in B(L^2(\Omega)).
$$
We are now going to construct a trace preserving $*$-automorphism $U\colon\M\to\M$
such that 
$T=T_\phi$ satisfies (\ref{2Dil1}).

Let $1$ denote the unit of $N$.
For any integer $m\geq 1$  and for any finite family $(x_k)_{k=-m}^{m}$ in
$N$, set
$$
\mathop{\otimes}_{k} x_k = \cdots 1\otimes 1\otimes x_{-m}\otimes\cdots\otimes
x_0\otimes x_1\otimes\ldots \otimes x_m\otimes 1\otimes 1 \cdots,
$$
where each $x_k$ is in position $k$. Let $\sigma\colon N^\infty\to N^\infty$ be the 
$*$-automorphism such that 
$$
\sigma\bigl(\mathop{\otimes}_{k} x_k \bigr)=\mathop{\otimes}_{k} x_{k-1}
$$
for all finite families $(x_k)_{k}$ in $N$.  It is plain that $\sigma$ is trace preserving.
We set 
$$
S=I_{B(L^2(\Omega))}\overline{\otimes}\sigma\colon\M\longrightarrow \M. 
$$

Consider $N_+= \mathop{\overline{\otimes}}_{k\geq 1} N$ and 
$N_-= \mathop{\overline{\otimes}}_{k\leq -1} N$, so that we can write
$$
N^\infty=N_-\overline{\otimes} N\overline{\otimes}N_+.
$$
This allows us to write
\begin{equation}\label{6Dec}
\M= N_-\overline{\otimes} \bigl(B(L^2(\Omega))\overline{\otimes} N\bigr)\overline{\otimes}N_+.
\end{equation}
Let $D\in L^\infty(\Omega)\overline{\otimes} N$ such that $d\sim D$
and recall that $L^\infty(\Omega)\overline{\otimes} 
N\subset B(L^2(\Omega))\overline{\otimes} N$,
by (\ref{4Mult}). 
Thus we may define
$$
\gamma\colon B(L^2(\Omega))\overline{\otimes} N\longrightarrow
B(L^2(\Omega))\overline{\otimes} N,\qquad \gamma(y)=D^*yD.
$$
Since $D$ is a unitary, $\gamma$ is a trace preserving $*$-automorphism. 
Owing to the decomposition (\ref{6Dec}), we can define
$$
\Gamma=I_{N_-}\overline{\otimes} \gamma\overline{\otimes}I_{N_+}\colon \M\longrightarrow \M.
$$

Now we set 
$$
U=\Gamma\circ S.
$$
By construction, $U\colon\M\to \M$ is a trace preserving $*$-automorphism. 

Let $u,v,a,b\in L^2(\Omega)$. On the one hand, by a simple induction argument, it follows from 
(\ref{2Trace}) that for all 
$k\geq 0 $, we have
$$
{\rm tr}\Bigl(\bigl[T_\phi^k(u\otimes v)\bigr]\bigl[a\otimes b\bigr]\Bigr)
=\int_{\Omega^2} \phi(s,t)^ku(t)v(s)a(s)b(t)\, dtds.
$$
On the other hand, using (\ref{1TraceForm}),
it follows from the arguments in the proof of \cite[Lemma 4.9]{Du}
and in the proof of \cite[Theorem A]{Du} coming right after it, that for all 
$k\geq 0 $, we have
\begin{equation}\label{6DilForm}
\tau_{\mathcal M}\Bigl(\bigl[U^k((u\otimes v)\otimes 1_\infty)\bigr]
\bigl[(a\otimes b)\otimes 1_\infty\bigr]\Bigr)
=\int_{\Omega^2} \phi(s,t)^k u(t)v(s)a(s)b(t)\, dtds.
\end{equation}
According to the definition of $J$, these identities imply that 
$$
\langle J_1^*U^kJ(u\otimes v),a\otimes b\rangle =\langle T_\phi^k(u\otimes v),a\otimes b\rangle.
$$
Since the linear span of all $u\otimes v$ as above is $w^*$-dense in $B(L^2(\Omega))$,
the linear span of all $a\otimes b$ as above is dense in $S^1(L^2(\Omega))$ and
$J,U$ and $T_\phi$ are $w^*$-continuous, this implies that 
$T=T_\phi$ satisfies (\ref{2Dil1}).

\begin{remark}\label{6FD} 
The reader not familiar with the approximation techniques  
in \cite{Du} used to establish (\ref{6DilForm}) should first look at the finite-dimensional case 
to gain intuition on this proof. We provide this finite-dimensional proof for convenience. Thus,
we now assume that $\Omega=\{1,\ldots,n\}$ for some integer $n\geq 1$, so that
$L^\infty(\Omega)=\ell^\infty_n$, $L^2(\Omega)=\ell^2_n$,
$B(L^2(\Omega))=M_n$ and $\M=M_n(N^\infty)$.
Then  $d\in\ell^n(N)$ is written as $d=(d_1,\ldots,d_n)$, where each $d_k\in N$
is a unitary, and $\phi=[\phi(i,j)]_{1\leq i,j\leq n}$ is given by
$$
\phi(i,j)=\tau_N(d_i^*d_j),\qquad 1\leq i,j\leq n.
$$
Further $\gamma\colon M_n(N)\to M_n(N)$ maps any $N$-valued matrix $[y_{ij}]$ to
$[d_i^*y_{ij}d_j]$.

Write $u=(u_1,\ldots,u_n)\in\ell^2_n$ and similarly for $v,a,b$.  
Then in the space $\M=M_n(N^\infty)$, we have $u\otimes v\otimes 1_\infty=[u_jv_i\cdotp 1_\infty]$.
Hence
$S(u\otimes v\otimes 1_\infty)
=(u\otimes v\otimes 1_\infty)$, hence
$$
U(u\otimes v\otimes 1_\infty) = \bigl[u_jv_i(
\cdots 1\otimes d_i^*d_j\otimes 1 \cdots)\bigr],
$$
where the $d_i^*d_j$ are in position $0$. Next, 
$$
S\bigl(U(u\otimes v\otimes 1_\infty)\bigr)=
\bigl[u_jv_i(
\cdots 1\otimes 1\otimes d_i^*d_j\otimes 1\cdots)\bigr],
$$
where the $d_i^*d_j$ are in position $1$ and
$$
U^2(u\otimes v\otimes 1_\infty) = \bigl[u_jv_i(
\cdots 1\otimes d_i^*d_j\otimes d_i^*d_j\otimes 1 \cdots)\bigr],
$$
where the $d_i^*d_j$ appear in positions $0$ and $1$. By induction, we find that
for all $k\geq 0$,
$$
U^k(u\otimes v\otimes 1_\infty) = \bigl[u_jv_i(
\cdots 1\otimes d_i^*d_j\otimes d_i^*d_j\otimes\cdots\otimes d_i^*d_j\otimes 1 \cdots)\bigr],
$$
where the $d_i^*d_j$ appear in positions $0,1,\ldots k-1$. 
We derive that 
$$
\langle U^k(u\otimes v\otimes 1_\infty),
a\otimes b\otimes 1_\infty\rangle_{{\mathcal M},{\mathcal M}_*} 
=\sum_{i,j=1}^n a_ib_ju_jv_i\tau_N(d_i^*d_j)^k,
$$
which is the requested identity (\ref{6DilForm}) in this context.
\end{remark}

\subsection{Complementary results}

\begin{remark}\label{5k1}
In the proof of the implication ``$(i) \Rightarrow (ii)$'' in Theorem \ref{1Main},
we use the identity (\ref{2Dil1}) only with $k=1$.  Consequently, the 
assertions (i) and (ii) in Theorem \ref{1Main} are also equivalent to:
\begin{itemize}
\item [(i)'] {\it There exist a tracial von Neumann algebra
$(\M,\tau_{\mathcal M})$ with a separable predual, a
$w^*$-continuous trace preserving unital
$*$-homomorphism $J\colon B(L^2(\Omega))\to \M$
and a trace preserving $*$-automomorphism $U\colon \M\to \M$ such that 
$T_\phi=J_1^* U J.$}
\end{itemize}
\end{remark}

\begin{remark}\label{3SA}
Let $\phi\in L^\infty(\Omega^2)$. It follows from Remark \ref{2Dil2}, (3), that if 
$\phi$ is a bounded Schur multiplier and 
$T_\phi\colon B(L^2(\Omega))\to B(L^2(\Omega))$
is absolutely dilatable, then $T_\phi$ is positive and unital. 
The converse is not true, that is, there exist positive unital Schur multiplier operators
which are not absolutely dilatable (see Remark \ref{5FD} below).
However it was proved in \cite[Theorem A]{Du} that if 
$T_\phi\colon B(L^2(\Omega))\to B(L^2(\Omega))$ is positive and unital,
and if $T_\phi\colon S^2(L^2(\Omega))\to S^2(L^2(\Omega))$ is
self-adjoint (equivalently, if $\phi$ is real-valued), then $T_\phi$
is absolutely dilatable. In the discrete case (see Corollary \ref{5Discrete}), 
the latter result is due to Arhancet \cite{Arh1}.
Furthermore, it follows from the proof of \cite[Theorem A]{Du}
that if $(\Omega,\mu)$ is {\it separable}, then any positive
unital self-adjoint $T_\phi\colon B(L^2(\Omega))\to B(L^2(\Omega))$
admits a separable absolute dilation.
\end{remark}

We now focus on the discrete case, that is 
$\Omega=\Ndb^*=\{1,2,\ldots\}$. An
element of $L^\infty(\Omega^2)$ is just a bounded family  ${\frak m}=\{m_{ij}\}_{i,j\geq 1}$
of complex numbers. Then ${\frak m}$ is a bounded Schur multiplier if and only if
there exists a constant $C\geq 0$ such that for all $n\geq 1$ and 
all $[a_{ij}]_{1\leq i,j\leq n}\in M_n$, 
$$
\bignorm{[m_{ij}a_{ij}]}_{M_n}\leq C 
\bignorm{[a_{ij}]}_{M_n}.
$$
Moreover for any von Neumann algebra $M$, we may identify
$L^\infty(\Omega)\overline\otimes M$
with $\ell^\infty(M)$,
the space of all bounded sequences with entries in  $M$. Likewise,
$L^\infty(\Omega)\overline\otimes L^\infty(\Omega)\overline\otimes M\simeq 
\ell^\infty_{{\mathbb N}^*\times{\mathbb N}^*}(M)$.
If $D=(d_k)_{k\geq 1}\in \ell^\infty(M)$, then $D$ is a unitary
if and only if $d_k$ is a unitary for all $k\geq 1$. Moreover,
the product $(1_j\otimes D_i^*)(1_i\otimes D_j)$ coincides with
the doubly indexed sequence $(d_i^* d_j)_{i,j\geq 1}$.
With these observations in mind, the following is a direct consequence of
Theorem \ref{1Main}.

\begin{corollary}\label{5Discrete}
For any bounded  family  ${\frak m}=\{m_{ij}\}_{i,j\geq 1}$
of complex numbers, the following assertions are equivalent.
\begin{itemize}
\item [(i)] The family ${\frak m}$ is a bounded Schur multiplier and 
$T_{\frak m}\colon B(\ell^2)\to B(\ell^2)$ admits a separable
absolute dilation.
\item [(ii)] There exist a normalized tracial von Neumann algebra $(N,\tau_N)$ 
with a separable predual
and a sequence $(d_k)_{k\geq 1}$ of unitaries of $N$ such that
$$
m_{ij} =\tau_N\bigl(d_i^*d_j\bigr), \qquad i,j\geq 1.
$$
\end{itemize}
\end{corollary}

\section{The non separable case}\label{S-NSC}
Let $(\Omega,\mu)$ be a $\sigma$-finite measure space.
If $X$ is a non separable Banach space, it is still possible to define
a space $L^\infty_\sigma(\Omega;X^*)$ of classes of
$w^*$-measurable functions $\Omega\to X^*$ in such a way that
$L^1(\Omega;X)^*\simeq L^\infty_\sigma(\Omega;X^*)$ isometrically, see 
\cite{Buk1, Buk2}. However these spaces are delicate to define and to use,
because if $d\colon\Omega\to X^*$ is a $w^*$-measurable function, then
the norm function $t\mapsto \norm{d(t)}$ may not be measurable.
Further given a von Neumann algebra $N$ and  
$w^*$-measurable functions $c,d\colon\Omega\to N$, the function
$c\odot d\colon\Omega^2\to N$ defined by (\ref{7cd}) is likely to
be non measurable. For these reasons, it seems to be out of reach to 
extend Theorem \ref{1Main} verbatim to the non separable case.

In this general setting, we will content ourselves with a version of 
Theorem \ref{1Main} stated in the language of von Neumann tensor 
products. In the next statement, we use some notations introduced in Section
\ref{Prep}, see in particular (\ref{Ref1}) and (\ref{Ref2}).
The most important implication here is ``(i)$\,\Rightarrow\,$(ii)$_1$", which
extends the implication ``(i)$\,\Rightarrow\,$(ii)" of Theorem \ref{1Main}
to the non separable case.

\begin{theorem}\label{3Main}
Let $\phi\in L^\infty(\Omega^2)$ and let $e\in L^2(\Omega)$
with $\norm{e}_2=1$. The following assertions are equivalent.
\begin{itemize}
\item [(i)] The function $\phi$ is a bounded Schur multiplier and 
$T_\phi\colon B(L^2(\Omega))\to B(L^2(\Omega))$
is absolutely dilatable.
\item [(ii)] There exist a normalized tracial von Neumann algebra $(N,\tau_N)$
acting on some Hilbert space $H$, 
and a unitary $D\in L^\infty(\Omega)\overline\otimes B(H)$
such that:
\begin{itemize}
\item [(ii)$_1$] The product $(1_t\otimes D_s^*)(1_s\otimes D_t)$
belongs to $L^\infty_s(\Omega)\overline{\otimes}L^\infty_t(\Omega)\overline{\otimes} N$
and under the identification (\ref{3Omega3}), we have
\begin{equation}\label{3TraceForm}
\phi= \bigl[I_{L^\infty(\Omega^2)}\overline{\otimes}\tau_N\bigr]
\bigl((1_t\otimes D_s^*)(1_s\otimes D_t)\bigr).
\end{equation}
\item [(ii)$_2$] The trace of $D^*(e\otimes e\otimes I_{H})D$
in $B(L^2(\Omega))\overline{\otimes} N$ is equal to $1$.
\end{itemize}
\end{itemize}
\end{theorem}

\begin{proof} 
\underline{$(i)\,\Rightarrow\,(ii)$}: Assume (i) and apply the construction devised in Subsection \ref{61}.  
Note that until Lemma \ref{4Trace}, this construction does not use separability.
Recall from Lemma  \ref{4D} that we are working with a unitary 
$D\in L^\infty(\Omega)\overline{\otimes} B(H_1,H_2)$.
Let $K$ be a (large enough) Hilbert space such that $H_1\stackrel{2}{\otimes} K$
is isomorphic to $H_2\stackrel{2}{\otimes} K$ and let
$L\colon H_1\stackrel{2}{\otimes} K\to H_2\stackrel{2}{\otimes} K$ be a unitary.
We define 
$$
\widetilde{D}=[I_{L^\infty(\Omega)}\overline{\otimes} L^*](D\otimes I_K)
\,\in L^\infty(\Omega)\overline{\otimes}
B(H_1\stackrel{2}{\otimes} K),
$$
where $D\otimes I_K$ is regarded as an element of  $L^\infty(\Omega)\overline{\otimes}
B(H_1\stackrel{2}{\otimes} K,H_2\stackrel{2}{\otimes} K)$. 
Clearly $\widetilde{D}$ is a unitary.
We will now check that the assertion (ii) of Theorem \ref{3Main} is satisfied
with the Hilbert space $H=H_1\stackrel{2}{\otimes} K$, the von Neumann
algebra 
$$
N=N_1\otimes I_K\subset B(H),
$$
the tracial state $\tau_N$ on
$N$ defined by $\tau_N(x\otimes I_K)=\tau_1(x)$ for all $x\in
N_1$ and the unitary
$\widetilde{D}$ in place of $D$.

First we have $(1_s\otimes \widetilde{D}_s^*)(1_s\otimes \widetilde{D}_t)
=(1_s\otimes {D}_s^*)(1_s\otimes {D}_t)\otimes I_K$,
because $LL^*=I_{H_2}\otimes I_K$. Hence by Lemma \ref{4Ess},
the product $(1_s\otimes \widetilde{D}_s^*)(1_s\otimes \widetilde{D}_t)$
belongs to $L^\infty_s(\Omega)\overline{\otimes}L^\infty_t(\Omega)\overline{\otimes} N$. Moreover 
we have
$$
\bigl[I_{L^\infty(\Omega^2)}\overline{\otimes}\tau_N\bigr]
\bigl((1_t\otimes \widetilde{D}_s^*)(1_s\otimes \widetilde{D}_t)\bigr)= 
\bigl[I_{L^\infty(\Omega^2)}\overline{\otimes}\tau_{N_1}\bigr]
\bigl((1_t\otimes D_s^*)(1_s\otimes D_t)\bigr).
$$
Applying Lemma \ref{4Trace}, we deduce that the assertion (ii)$_1$ is satisfied.

Next we note, using again  $LL^*=I_{H_2}\otimes I_K$, that
$$
\widetilde{D}^*(e\otimes e\otimes I_H)\widetilde{D}
=D^*(e\otimes e\otimes I_{H_2})D\otimes I_K
=\pi^{-1}(e\otimes e\otimes 1_{N_2}) \otimes I_K,
$$
where $\pi$ comes from (\ref{4pi0}). Since $\pi$ is trace preserving,
the trace of  $\widetilde{D}^*(e\otimes e\otimes I_H)\widetilde{D}$
in $B(L^2(\Omega))\overline{\otimes} N$ is equal to the trace
of $e\otimes e\otimes 1_{N_2}$ in  $B(L^2(\Omega))\overline{\otimes} N_2$.
The latter is equal to $1$ (recall that $\tau_2$ is normalized).
Hence the assertion (ii)$_2$ is satisfied.

\smallskip
\underline{$(ii)\,\Rightarrow\,(i)$}:
Assume (ii) and consider the $*$-automorphism 
$$
\pi\colon B(L^2(\Omega))\overline{\otimes} B(H)\longrightarrow
B(L^2(\Omega))\overline{\otimes} B(H),
\qquad \pi(y)=DyD^*.
$$
Let $B=\pi\bigl(B(L^2(\Omega))\overline{\otimes}N\bigr)$. This
is a von Neumann subalgebra
of $B(L^2(\Omega))\overline{\otimes} B(H)$. We equip it with the trace $\tau_B
=({\rm tr}\overline{\otimes} \tau_N)\circ\pi^{-1}$, so that the restriction 
$B(L^2(\Omega))\overline{\otimes}N\to B$ of $\pi$
is trace preserving by nature.

For any $u,v\in L^2(\Omega)$, $D^*(u\otimes v\otimes I_H)D$ belongs to 
$B(L^2(\Omega))\overline{\otimes} N$, by the assumption (ii)$_1$ and
Lemma \ref{4N}. (Note that this is the reason why it makes sense to
consider $D^*(e\otimes e\otimes I_{H})D$ as an element of 
$B(L^2(\Omega))\overline{\otimes} N$ in (ii)$_2$.)
Since $D$ is a unitary, we have
$u\otimes v\otimes I_H =\pi(D^*(u\otimes v\otimes I_H)D)$, hence $u\otimes v\otimes I_H$
belongs to $B$. Since the linear span of all the $u\otimes v$ as
above is $w^*$-dense in $B(L^2(\Omega))$, we deduce that
$$
B(L^2(\Omega))\otimes I_H\subset B\subset 
B(L^2(\Omega))\overline{\otimes} B(H).
$$
Let $\{E_{ij}\,:\, (i,j)\in I^2\}$ be a system of matrix units
in $B(L^2(\Omega))$. Then it follows from above
that $\{E_{ij}\otimes I_H \,:\,(i,j)\in I^2\}$ is a system of matrix units
in $B$. Using again \cite[Proposition IV.1.8]{Tak}
and its proof, we obtain the existence of a von Neumann
subalgebra $M\subset B(H)$ such that $B=B(L^2(\Omega))\overline{\otimes} M$.
According to the assumption (ii)$_2$ and the definition of $\tau_B$, we have
$$
\tau_B(e\otimes e\otimes 1_M) = ({\rm tr}\overline{\otimes} \tau_N)
\bigl(D^*(e\otimes e\otimes I_H)D\bigr)=1.
$$
This implies that $\tau_M\colon m\mapsto \tau_B(e\otimes e\otimes m)$ is a well-defined
tracial state on $M$.  A classical uniqueness argument shows that
$\tau_B={\rm tr}\overline{\otimes} \tau_M$.

We can now adapt the proof given in Subsection \ref{62}, as follows.
We set
$$
M_- = \mathop{\overline{\otimes}}_{k\leq -1} M
\qquad\hbox{and}\qquad N_+ = \mathop{\overline{\otimes}}_{k\geq 1} N.
$$
More precisely, $M_-$  is the infinite tensor product of $(M,\tau_M)$
over the index set $\{-1,-2,\ldots\}$ whereas $N_+$  is the infinite tensor product of $(N,\tau_N)$
over the index set $\{1,2,\ldots\}$. Then we define
$$
\M=B(L^2(\Omega))\overline{\otimes}M_-\overline{\otimes}M\overline{\otimes} N\overline{\otimes}N_+,
$$
that we equip with its natural trace.
We define $J\colon B(L^2(\Omega))\to \M$ by $J(z)=z\otimes 1$, where $1$ is the unit of
$M_-\overline{\otimes}M\overline{\otimes} N\overline{\otimes}N_+$.
We set
$$
\M_1 = M_-\overline{\otimes}\bigl(B(L^2(\Omega))
\overline{\otimes}M\bigr)\overline{\otimes}N\overline{\otimes}N_+
\qquad\hbox{and}\qquad
\M_2 =  M_-\overline{\otimes}M\overline{\otimes}
\bigl(B(L^2(\Omega))
\overline{\otimes}N\bigr)\overline{\otimes}N_+,
$$
and we let 
$\kappa_1\colon\M\to\M_1$ and $\kappa_2\colon\M\to\M_2$
be the canonical $*$-automorphisms.  We let
$\sigma_-\colon M_-\to M_-\overline{\otimes}M$ be the $*$-isomorphism 
such that
$$
\sigma_-\Bigl(\mathop{\otimes}_{k\leq -1} x_k \Bigr)=\Bigl(\mathop{\otimes}_{k\leq -1} x_{k-1}\Bigr)\otimes x_{-1},
$$
for all finite families $(x_k)_{k\leq -1}$ in $M$.  Likewise, we let $\sigma_+\colon N\overline{\otimes} N_+
\to  N_+$ be the $*$-isomorphism 
such that
$$
\sigma_+\Bigl(x_0\otimes\Bigl(\mathop{\otimes}_{k\geq 1} x_k \Bigr)\Bigr)=\mathop{\otimes}_{k\geq 1} x_{k-1},
$$
for all finite families $(x_k)_{k\geq 0}$ in $N$.  These two $*$-isomorphisms are trace preserving.

It follows from the construction of $M$ that we may define a
$*$-isomorphism
$$
\gamma\colon B(L^2(\Omega))
\overline{\otimes}M\longrightarrow B(L^2(\Omega))
\overline{\otimes}N,
\qquad \gamma(y)=D^*yD.
$$
This is the inverse of the restriction  of $\pi$ to
$B(L^2(\Omega))\overline{\otimes}N$, hence it is trace preserving.

Finally, we define $U\colon\M\to \M$ by
$$
U = \kappa_2^{-1}\circ\bigl(\sigma_-\overline{\otimes}\gamma\overline{\otimes}\sigma_+\bigr)\circ\kappa_1.
$$
Then $U$ is a trace preserving $*$-automorphism. Next arguing again
as in \cite{Du}, we obtain that (\ref{6DilForm}) holds true for all
$u,v,a,b\in L^2(\Omega)$. As in Subsection \ref{62}, we deduce that
$T=T_\phi$ satisfies (\ref{2Dil1}). 
\end{proof}

We finally go back to the discrete case and give a complement to 
Corollary \ref{5Discrete}.

\begin{corollary}\label{5Discrete2}
For any bounded  family  ${\frak m}=\{m_{ij}\}_{i,j\geq 1}$
of complex numbers, the following assertions are equivalent.
\begin{itemize}
\item [(i)] The family ${\frak m}$ is a bounded Schur multiplier and 
$T_{\frak m}\colon B(\ell^2)\to B(\ell^2)$
is absolutely dilatable.
\item [(ii)] There exist a normalized tracial von Neumann algebra $(N,\tau_N)$ 
and a sequence $(d_k)_{k\geq 1}$ of unitaries of $N$ such that
$$
m_{ij} =\tau_N\bigl(d_i^*d_j\bigr), \qquad i,j\geq 1.
$$
\item [(iii)] The family ${\frak m}$ is a bounded Schur multiplier and 
$T_{\frak m}\colon B(\ell^2)\to B(\ell^2)$ admits a separable
absolute dilation.
\end{itemize}
\end{corollary}

\begin{proof}
Assume (i) and apply Theorem \ref{3Main} to ${\frak m}$. 
We find a normalized tracial von Neumann algebra $(N,\tau_N)$ 
acting on some $H$ and a sequence $(d_k)_{k\geq 1}$ of unitaries of $B(H)$
such that $d_i^*d_j\in N$ and
$m_{ij} =\tau_N\bigl(d_i^*d_j\bigr)$ for all $i,j\geq 1$.
Set $d'_k=d_1^*d_k$, $k\geq 1$. The $d'_k$ are unitaries
of $N$ and $d_i'^*d_j'=d_i^*d_j$ for all $i,j\geq 1$. 
Hence $m_{ij} =\tau_N\bigl(d_i'^*d_j'\bigr)$, for all $i,j\geq 1 $. This
proves (ii).

Assume (ii) and let $\widetilde{N}\subset N$ be the von Neumann subalgebra 
generated by $(d_k)_{k\geq 1}$. Then $\widetilde{N}$ admits 
a sequence $(a_n)_{n\geq 1}$ such that
$V:={\rm Span}\{a_n\, :\, n\geq 1\}$ is $w^*$-dense in $\widetilde{N}$. 
Let $\tau_{\widetilde{N}}$ be the restriction
of $\tau_N$ to $\widetilde{N}$ and recall that we have a contractive embedding
$\widetilde{N}\subset L^1(\widetilde{N},\tau_{\widetilde{N}})$.
If $b\in \widetilde{N}$ is such that $\tau_{\widetilde{N}}(a_nb)=0$ for all $n\geq 1$,
then as an element of $L^1(\widetilde{N},\tau_{\widetilde{N}})$, $b$ belongs to $V_\perp$.
Hence $b=0$. In turn, this shows that $V$ is dense in
$L^1(\widetilde{N},\tau_{\widetilde{N}})$. Thus, $\widetilde{N}$ has a separable predual.
By  Corollary \ref{5Discrete},  the family ${\frak m}$ therefore
satisfies (iii).
The remaining assertion $``(iii)\,\Rightarrow\,(i)"$ is obvious.
\end{proof}

\begin{remark}\label{5FD}
In the finite dimensional case, we find similarly that the Schur multiplier operator $T_{\frak m}\colon M_n\to M_n$
associated with a family ${\frak m}=\{m_{ij}\}_{1\leq i,j\leq n}$ admits a separable absolute dilation
if and only if it admits an absolute dilation, if and only if
there exist a normalized tracial von Neumann algebra $(N,\tau_N)$ 
and   unitaries $d_1,\ldots,d_n$ in $N$ such that
$m_{ij} =\tau_N\bigl(d_i^*d_j\bigr)$ for all $1\leq i,j\leq n$.
The latter equivalence result is implicit in \cite{HM}. Indeed it follows from Proposition
2.8 and Theorem 4.4 in this reference.

Taking this equivalence into account, \cite[Example 3.2]{HM} exhibits a positive unital Schur multiplier 
$M_4\to M_4$ which is not absolutely dilatable.

We also mention that there are plenty of absolutely dilatable Schur multiplier
operators $T_{\frak m}\colon M_n\to M_n$ such that ${\frak m}$ is not real valued.
For example, it follows from above that for any complex number $\omega$ with
$\vert\omega\vert =1$, the operator $T_{\frak m}\colon M_2\to M_2$
associated with
$${\frak m}=\left[\begin{matrix} 1 & \omega \\ \overline{\omega} & 1\end{matrix}\right]$$
is absolutely dilatable.
\end{remark}

\smallskip\noindent
{\bf Acknowledgements.} Both authors are supported by the ANR project 
{\it Noncommutative analysis on groups and quantum groups} 
(No./ANR-19-CE40-0002). We thank the referee for several comments which 
improved the introduction of the paper.

\bibliographystyle{abbrv}

\vskip 0.3cm

\end{document}